\documentclass[a4paper,12pt]{amsart}
\usepackage{amsmath,amsthm,amssymb,latexsym,enumerate,color,hyperref}
\usepackage[margin=2.2cm]{geometry}
\usepackage{graphicx}
\usepackage[alphabetic]{amsrefs}

\usepackage{xcolor}

\numberwithin{equation}{section}

\begin{document}

\newtheorem{thm}{Theorem}[section]
\newtheorem{prop}[thm]{Proposition}
\newtheorem{lem}[thm]{Lemma}
\newtheorem{cor}[thm]{Corollary}
\newtheorem{rem}[thm]{Remark}
\newtheorem*{defn}{Definition}

\newcommand{\DD}{\mathbb{D}}
\newcommand{\NN}{\mathbb{N}}
\newcommand{\ZZ}{\mathbb{Z}}
\newcommand{\QQ}{\mathbb{Q}}
\newcommand{\RR}{\mathbb{R}}
\newcommand{\CC}{\mathbb{C}}
\renewcommand{\SS}{\mathbb{S}}

\newcommand{\supp}{\mathop{\mathrm{supp}}}    

\newcommand{\re}{\mathop{\mathrm{Re}}}   
\newcommand{\im}{\mathop{\mathrm{Im}}}   
\newcommand{\dist}{\mathop{\mathrm{dist}}}  
\newcommand{\link}{\mathop{\circ\kern-.35em -}}
\newcommand{\spn}{\mathop{\mathrm{span}}}   
\newcommand{\ind}{\mathop{\mathrm{ind}}}   
\newcommand{\rank}{\mathop{\mathrm{rank}}}   
\newcommand{\Fix}{\mathop{\mathrm{Fix}}}   
\newcommand{\codim}{\mathop{\mathrm{codim}}}   
\newcommand{\conv}{\mathop{\mathrm{conv}}}   
\newcommand{\epsi}{\mbox{$\varepsilon$}}
\newcommand{\eps}{\mathchoice{\epsi}{\epsi}
{\mbox{\scriptsize\epsi}}{\mbox{\tiny\epsi}}}
\newcommand{\cl}{\overline}
\newcommand{\pa}{\partial}
\newcommand{\ve}{\varepsilon}
\newcommand{\zi}{\zeta}
\newcommand{\Si}{\Sigma}
\newcommand{\cA}{{\mathcal A}}
\newcommand{\cG}{{\mathcal G}}
\newcommand{\cH}{{\mathcal H}}
\newcommand{\cI}{{\mathcal I}}
\newcommand{\cJ}{{\mathcal J}}
\newcommand{\cK}{{\mathcal K}}
\newcommand{\cL}{{\mathcal L}}
\newcommand{\cN}{{\mathcal N}}
\newcommand{\cR}{{\mathcal R}}
\newcommand{\cS}{{\mathcal S}}
\newcommand{\cT}{{\mathcal T}}
\newcommand{\cU}{{\mathcal U}}
\newcommand{\OM}{\Omega}
\newcommand{\B}{\bullet}
\newcommand{\ol}{\overline}
\newcommand{\ul}{\underline}
\newcommand{\vp}{\varphi}
\newcommand{\AC}{\mathop{\mathrm{AC}}}   
\newcommand{\Lip}{\mathop{\mathrm{Lip}}}   
\newcommand{\es}{\mathop{\mathrm{esssup}}}   
\newcommand{\les}{\mathop{\mathrm{les}}}   
\newcommand{\nid}{\noindent}
\newcommand{\pzr}{\phi^0_R}
\newcommand{\pir}{\phi^\infty_R}
\newcommand{\psr}{\phi^*_R}
\newcommand{\pow}{\frac{N}{N-1}}
\newcommand{\ncl}{\mathop{\mathrm{nc-lim}}}   
\newcommand{\nvl}{\mathop{\mathrm{nv-lim}}}  
\newcommand{\la}{\lambda}
\newcommand{\La}{\Lambda}    
\newcommand{\de}{\delta}    
\newcommand{\fhi}{\varphi} 
\newcommand{\ga}{\gamma}    
\newcommand{\ka}{\kappa}   

\newcommand{\core}{\heartsuit}
\newcommand{\diam}{\mathrm{diam}}

\newcommand{\lan}{\langle}
\newcommand{\ran}{\rangle}
\newcommand{\tr}{\mathop{\mathrm{tr}}}
\newcommand{\diag}{\mathop{\mathrm{diag}}}
\newcommand{\dv}{\mathop{\mathrm{div}}}

\newcommand{\al}{\alpha}
\newcommand{\be}{\beta}
\newcommand{\Om}{\Omega}
\newcommand{\na}{\nabla}

\newcommand{\cC}{\mathcal{C}}
\newcommand{\cM}{\mathcal{M}}
\newcommand{\nr}{\Vert}
\newcommand{\De}{\Delta}
\newcommand{\cX}{\mathcal{X}}
\newcommand{\cP}{\mathcal{P}}
\newcommand{\om}{\omega}
\newcommand{\si}{\sigma}
\newcommand{\te}{\theta}
\newcommand{\Ga}{\Gamma}

\title[A Serrin-type problem with partial knowledge of the domain]{A Serrin-type problem with \\ partial knowledge of the domain}

\author{Serena Dipierro}
\address{Serena Dipierro: Department of Mathematics and Statistics, The University of Western Australia, 35 Stirling Highway, Crawley, Perth, WA 6009, Australia}
\email{serena.dipierro@uwa.edu.au}

\author{Giorgio Poggesi}
\address{Giorgio Poggesi: Department of Mathematics and Statistics, The University of Western Australia, 35 Stirling Highway, Crawley, Perth, WA 6009, Australia}
\email{giorgio.poggesi@uwa.edu.au}

\author{Enrico Valdinoci}
\address{Enrico Valdinoci: Department of Mathematics and Statistics, The University of Western Australia, 35 Stirling Highway, Crawley, Perth, WA 6009, Australia}
\email{enrico.valdinoci@uwa.edu.au}
   
\dedicatory{To Matilde Aurora Alessi, for inspiring us with a question on heating devices}

\begin{abstract}
We present a quantitative estimate for the radially symmetric configuration
concerning a Serrin-type overdetermined problem for the torsional rigidity
in a bounded domain $\Om \subset \RR^N$, when the equation is known
on~$\Om \setminus \ol{\om}$ only, for some open subset~$\om \Subset \Om$.

The problem has concrete motivations in optimal heating with
malfunctioning, laminar flows and beams with small inhomogeneities. 
\end{abstract}

\keywords{Serrin's overdetermined problem, torsional rigidity, integral identities, stability, quantitative estimates}
\subjclass[2010]{Primary 35N25, 53A10, 35B35; Secondary 35A23}

\maketitle

\raggedbottom

\section{Introduction}
In this article we consider a variation of the classical Serrin's
overdetermined problem~\cite{Se} in which the equation is only known in a subset of the domain. We will provide quantitative
results showing, roughly speaking, that when
the part of the domain in which we do not have information is ``small'',
then the domain is ``close'' to a ball.

\subsection{Statement of the problem and main result}

The precise problem that we consider may be stated as follows.
Let $\Om \subset \RR^N$ be a bounded domain -- that is a bounded,
open, connected set, whose
boundary will be denoted by~$\Ga$,
and let $\om \Subset \Om$ be an open
(not necessarily connected) subset of $\Om$ with boundary denoted by~$\pa \om$.
We consider the following problem:
\begin{equation}\label{eq:problem}\begin{cases}
\De u = 1 \quad\text{  in  } \Om \setminus \ol{\om}, \\ u=0 \quad\text{  on  } \Ga, \end{cases}
\end{equation}
under the overdetermined condition
\begin{equation}\label{eq:overdetermination}
u_{\nu} = c \quad\text{  on  } \Ga ,
\end{equation}
for some~$c\in\RR$.
Here and in what follows $\nu$ denotes the outward unit normal of $\Om \setminus \ol{\om}$ and $u_\nu$ the derivative of $u$ in the direction $\nu$. Concerning the setting in~\eqref{eq:overdetermination}, we remark that even without explicitly imposing any
regularity assumptions on $\Ga$, \cite[Theorem 1]{Vo} guarantees that 
\begin{equation}\label{noWea}
\begin{split}&
{\mbox{if \eqref{eq:overdetermination} holds true (in the appropriate weak sense)}}\\&{\mbox{then $\Ga$ is of class $C^{2,\al}$, with~$0 < \al \le 1$,}}\end{split}\end{equation}
therefore the notation~$u_\nu$ on $\Ga$ is well posed in the classical sense, being 
$u \in C^{2,\al} \left( \left( \Om \setminus \ol{ \om} \right) \cup \Ga \right)$ by standard elliptic regularity theory.
%
%

We will further assume $u$ to be of class $C^2$ up to $\pa \om$, and hence
\begin{equation}\label{assumption:regularityuptobordino}
u \in C^2 (\ol{\Om} \setminus \om)
\end{equation}
(see Section~\ref{NOATZ} for details on this).

Concerning the regularity of the domain taken into account, to avoid unessential technicalities we will first assume that
\begin{equation}\label{A3}
{\mbox{$\pa \om$ is of class $C^1$}}\end{equation}
and that
\begin{equation}\label{A4}\begin{split}&
{\mbox{$\Om \setminus \ol{ \om }$ satisfies the ``uniform interior sphere condition'',}}\\&{\mbox{i.e. there exists $r_i>0$ such that for each $p \in \Ga \cup \pa \om $}}\\&{\mbox{there exists a ball contained in $\Om \setminus \ol{\om}$ of radius $r_i$}}\\&{\mbox{such that
its closure intersects $\Ga \cup \pa \om$ only at $p$.}}\end{split}\end{equation}
We recall, for instance, that domains with~$C^{1,1}$ boundaries
satisfy~\eqref{A4}, see e.g.~\cite[Lemma~A.1]{ROV}.

To state our main result we introduce
some notation.
For a given domain $D \subset \RR^N$, we denote by $|D|$ and $|\pa D|$
the $N$-dimensional Lebesgue measure of $D$
and the surface measure of $D$, respectively.
Our main result aims at considering a convenient point~$z$
and at obtaining suitable bounds
on the ``pseudo-distance''
\begin{equation}\label{INTRO:eq:pseudodistance}
\int_{\Ga} \left| \frac{|x-z|}{N} - c \right|^2 dS_x 
\end{equation}
and on the
``asymmetry''
\begin{equation}\label{INTRO:eq:asymmetry}
\frac{| \Om \De B_{Nc} (z) |}{|B_{Nc} (z)|} ,
\end{equation}
where $\Om \De B_{Nc} (z)$ denotes the symmetric
difference of $\Om$ and the ball $B_{Nc} (z)$ of radius $Nc$
centered at $z$.
In addition, we provide a ``geometric'' bound
on the set by estimating the difference between
the largest ball centered at~$z$ contained in~$\Omega$
and the smallest ball
centered at~$z$ that contains~$\Omega$.

The precise result that we have here goes as follows:

\begin{thm}\label{MAIN}
Let $\Om \setminus \ol{\om} \subset \RR^N$ be a bounded domain satisfying assumptions \eqref{A3} and \eqref{A4}.
Let $u$ satisfy \eqref{eq:problem}, \eqref{eq:overdetermination}
and~\eqref{assumption:regularityuptobordino}.
Assume that~$u \le 0$ on $\pa \om$.
Set
\begin{equation}\label{eq:introstaementez}
z:= \frac{1}{|\Om \setminus \ol{\om}|} \left\lbrace \int_{\Om \setminus \ol{\om}} x \, dx - N \int_{\pa \om} u\;\nu \, dS_x \right\rbrace.
\end{equation}
Then, 
\begin{equation}\label{INTRO:eq:pseudodistancestabconnormaC^2}
\int_{\Ga}  \left| \frac{|x-z|}{N} - c \right|^2 dS_x \le C | \pa \om | 
\end{equation}
and
\begin{equation}\label{INTRO:eq:asymmetrystabconnormaC^2}
\frac{| \Om \De B_{Nc} (z) |}{|B_{Nc} (z)|} \le C | \pa \om |^{1/2} .
\end{equation}
In \eqref{INTRO:eq:pseudodistancestabconnormaC^2} and \eqref{INTRO:eq:asymmetrystabconnormaC^2},
the quantity~$C$ denotes a positive constant depending only
on the dimension~$N$, on the internal radius~$r_i$
in~\eqref{A4}, on the diameter
of~$\Omega$, and on~$\|u\|_{C^2(\partial\omega)}$.

Moreover, if
\begin{equation}\label{BALB}
z \in \Om,\end{equation} then there exist~$\rho_e \ge \rho_i>0$
such that
\begin{equation}\label{CSL:1} B_{\rho_i}(z) \subset \Omega\subset B_{\rho_e}(z)\end{equation}
and
\begin{equation}\label{CSL:2}
\rho_e-\rho_i\le C \,|\pa \om|^{\tau_N/2},
\end{equation}
where
$C$ is a positive
constant depending only
on the dimension~$N$, on the internal radius~$r_i$
in~\eqref{A4}, on the diameter of~$\Omega$ and on~$\|u\|_{C^2(\partial\omega)}$, and~$\tau_N>0$ is a constant depending only on~$N$.
\end{thm}

We point out that 
Theorem~\ref{MAIN} collects {\em three different aspects}
on the problem: indeed,
the statement in~\eqref{INTRO:eq:pseudodistancestabconnormaC^2}
is an ``integral'' stability result in~$L^2$
(see Section \ref{sec:Some estimates}), 
the estimate in~\eqref{INTRO:eq:asymmetrystabconnormaC^2}
provides a stability ``in measure'', and 
the result in~\eqref{CSL:2} deals with a ``pointwise'' notion
of stability. 
%
%

\medskip

The exponent~$\tau_N$ in \eqref{CSL:2} can be better
precised. Indeed, we can take
$$ \tau_2 := 1.$$
Moreover, $\tau_3$ can be taken arbitrarily close to one, in the sense that
for any $\theta>0$, we have that~\eqref{CSL:2} holds with $\tau_3 := 1- \theta$
and~$C$ depending also on~$\theta$.

Furthermore, for $N \ge 4$, we can take
$$\tau_N :=\frac{ 2}{N-1}.$$ 

We think that it is an interesting open problem
to detect whether these choices of exponents~$\tau_2$, $\tau_3$,
$\tau_N$ in \eqref{CSL:2}, as well as the exponents appearing in \eqref{INTRO:eq:pseudodistancestabconnormaC^2} and \eqref{INTRO:eq:asymmetrystabconnormaC^2},
are optimal in Theorem~\ref{MAIN}. It would also 
be very interesting to have explicit examples to check
the optimality of the structural assumptions in Theorem~\ref{MAIN}.
\medskip

We also point out that condition~\eqref{BALB} is naturally satisfied in many
concrete situations (see also Remark~\ref{REMA72}):
in particular, for ``small'' $\om$, the point~$z$ is ``close''
to the baricenter of~$\Om$, hence condition~\eqref{BALB} is fulfilled
in this case when the baricenter of~$\Om$ lies in~$\Om$ (as it happens,
for instance, for convex sets).
\medskip

Of course, when~$\omega=\varnothing$, 
we have that~\eqref{CSL:2} reduces to~$\rho_e=\rho_i$
and therefore~\eqref{CSL:1} gives that~$\Omega$ is a ball:
in this sense Theorem~\ref{MAIN} recovers the classical
results of~\cites{Se, We} for overdetermined problems.
The main difference here is that, differently from the existing
literature, the equation is supposed to hold possibly only
outside a subdomain~$\omega$: as a counterpart,
Theorem~\ref{MAIN} does not prove that the full domain is a ball,
but only that it is geometrically close to a ball whenever the
subdomain~$\omega$
has a small Lebesgue measure.\medskip

In the present setting, Theorem~\ref{MAIN} will be in fact a particular
case of more general quantitative
results, presented in Section~\ref{SEC-6}, and relying on a number
of auxiliary integral identities. For more details, see Theorems \ref{thm:stability_radii}, \ref{thm:stab_pseudodistance}, \ref{thm:stab_Asymmetry}, and~\ref{JAHS334}.\medskip

In this spirit, Theorem~\ref{MAIN} falls within the broad stream
of research aiming at obtaining quantitative rigidity results,
see e.g.~\cites{ABR, BNST, CMV, CV1, Fe, Ma, MP1, MP2, MP3, Pog2}
and the references therein. More generally, overdetermined problems have been also considered
e.g. in~\cites{AB, CV2, EP, FV1, FV2, FV3, FV4, FG, FGK, FGLP,
MR1808686, MR2002730, GL, GS, GX, PS, Pog} and in
the references therein.

\subsection{Comments on the structural assumptions and generalizations}

We remark that assumption \eqref{A4} gives a lower bound on the distance not only between $\pa \om$ and $\Ga$ (being $\dist( \pa \om , \Ga) \ge 2 r_i$), but also between the boundaries of the  different connected components of $\om$ (if any).

Interestingly,
assumptions \eqref{A3} and \eqref{A4} can be relaxed, as explained in Section \ref{sec:relaxing assumptions}.

In particular, suitable
counterparts of \eqref{INTRO:eq:pseudodistancestabconnormaC^2} and \eqref{INTRO:eq:asymmetrystabconnormaC^2}
can be obtained if~\eqref{A3} and \eqref{A4} are replaced
by the weaker assumptions 
\begin{equation}\label{A5JOHN}
{\mbox{$\Om \setminus \ol{\om}$ is a John domain}}
\end{equation}
and
\begin{equation}\label{A3bis}
{\mbox{$\Om \setminus \ol{\om}$ is of finite perimeter.}}
\end{equation}
Moreover, a counterpart of the pointwise estimate \eqref{CSL:2} can be obtained if~\eqref{A3} and \eqref{A4} are dropped and replaced
by \eqref{A5JOHN}, \eqref{A3bis}, and the assumption 
\begin{equation}\label{A4bis}
\begin{split}&
{\mbox{$\Om \setminus \ol{ \om }$ satisfies the ``uniform interior sphere condition only
on $\Ga$'',}}\\&{\mbox{i.e.,
there exists $r_i>0$ such that for each $p \in \Ga $ there exists}}\\&{\mbox{a ball contained in $\Om \setminus \ol{\om}$ of radius $r_i$ such that its closure}}\\&{\mbox{intersects $\Ga \cup \pa \om$ only at $p$ on $\Ga$.}}\end{split} 
\end{equation}

We stress that \eqref{A4bis} is equivalent to assume a lower bound only
for $\dist(\pa \om , \Ga)$. Indeed, being $\Ga$ of class $C^{2, \al}$,
the set $\Om$ surely satisfies the uniform interior sphere condition on $\Ga$.


In this situation, Theorem~\ref{MAIN} remains valid, with the following structural
modifications:
\begin{itemize}
\item The boundary measure of~$\pa\om$ is replaced by its perimeter,
namely by the $(N-1)$-dimensional
Hausdorff measure $\cH^{N-1} (\pa^* \om)$ of its reduced boundary~$\pa^*\om$. In turn, $\nu$ on $\pa^* \om$ has to be intended as the (measure-theoretic)
outer unit normal (see Section~\ref{sec:relaxing assumptions}).
\item The constants $C$ depend on $r_i$ defined in \eqref{A4bis} and on the structural constant~$b_0$
of the given $b_0$-John domain;
\item The explicit expression of the exponents~$\tau_N$ in the pointwise estimate \eqref{CSL:2} is possibly worse than the ones obtained in Theorem~\ref{MAIN}.
\end{itemize}

The definition and details for $b_0$-John domains can be found in Subsection \ref{subsec:John}.
Here, we just stress that the class of John domains is huge: in particular, if \eqref{A4} is satisfied then $\Om \setminus \ol{\om}$ is surely a $b_0$-John domain with
$b_0 \le d_\Om /r_i$ (see \cite[(iii) of Remark 3.12]{Pog2}).

The precise statement that we have in this framework is stated next, and can be deduced from more general results presented in Theorems \ref{thm:relaxed_stab_pseudodistance}, \ref{thm:relaxed_stab_Asymmetry}, \ref{thm:Johnrelaxed_radiistability} and \ref{J:l90k1234cniw}.

\begin{thm}\label{MAIN-DUE}
Let~$\Om\setminus\overline\om \subset \RR^N$ be a bounded domain satisfying assumptions \eqref{A5JOHN} and \eqref{A3bis}.
Let $u$ satisfy \eqref{eq:problem}, \eqref{eq:overdetermination}
and~\eqref{assumption:regularityuptobordino}.
Assume that~$u \le 0$ on $\pa \om$.
Set
\begin{equation*}
z:= \frac{1}{|\Om \setminus \ol{\om}|} \left\lbrace \int_{\Om \setminus \ol{\om}} x \, dx - N \int_{\pa^* \om} u\;\nu \, d \cH^{N-1} \right\rbrace.
\end{equation*}
Then, the pseudodistance defined in \eqref{INTRO:eq:pseudodistance} and the asymmetry defined in \eqref{INTRO:eq:asymmetry} satisfy
\begin{equation}\label{INTRO:eq:generapseudodistC2normajohndomain}
\int_{\Ga}  \left| \frac{|x-z|}{N} - c \right|^2 \, d\cH^{N-1} \le C {\mathcal{H}}^{N-1}(\pa^* \om ),
\end{equation}
\begin{equation}\label{INTRO:eq:relaxed_asymmetrystabconnormaC^2}
\frac{| \Om \De B_{Nc} (z) |}{|B_{Nc} (z)|} \le C {\mathcal{H}}^{N-1}(\pa^* \om )^{1/2},
\end{equation}
where the constants $C>0$ appearing in \eqref{INTRO:eq:generapseudodistC2normajohndomain} and \eqref{INTRO:eq:relaxed_asymmetrystabconnormaC^2} depend only on $N$, $b_0$, $d_\Om$, $c$, and $\|u\|_{C^2(\partial\omega)}$.

If in addition \eqref{A4bis} is verified and $z \in \Om$, then there exist~$\rho_e \ge \rho_i>0$
such that
\begin{equation*} B_{\rho_i}(z)\subset \Omega\subset B_{\rho_e}(z)\end{equation*}
and
\begin{equation}\label{INTRO:eq:CSLGENERAL}
\rho_e-\rho_i\le C \,\big({\mathcal{H}}^{N-1}(\pa^* \om)\big)^{\tau_N/2},
\end{equation}
where
%
%
$C>0$ is a constant depending only
on $N$, the internal radius~$r_i$
in~\eqref{A4bis}, $d_\Om$, $b_0$, and
on~$\|u\|_{C^2(\partial\omega)}$. 
The exponents~$\tau_N>0$ depend only on~$N$.
\end{thm}

We point out that in Theorem~\ref{MAIN-DUE} we maintained the ``same''
choice \eqref{eq:introstaementez} for the point $z$.

The dependence of the constants $C$ on $c$ in \eqref{INTRO:eq:generapseudodistC2normajohndomain} and \eqref{INTRO:eq:relaxed_asymmetrystabconnormaC^2} could be replaced with the dependence on the surface measure $| \Ga |$, as explained in Remark \ref{rem:PROVA TOGLIERE C CON CONDIZIONE SU PERIMETRO}.

The explicit values of~$\tau_N$ in \eqref{INTRO:eq:CSLGENERAL} are the following ones.
We have that~$\tau_2$ can be taken as close as we wish to~$1$,
namely one can fix any~$\theta>0$ and take~$\tau_2:=1-\theta$
(in this case, the constant~$C$ in \eqref{INTRO:eq:CSLGENERAL} will also depend on~$\theta$).
When~$N\ge3$, one can take~$\tau_N:=\frac{2}{N}$.

We notice that these exponents are all smaller (i.e.,
``worse'')
than the ones obtained in Theorem~\ref{MAIN}.
Nevertheless, it is possible to get the pointwise estimate \eqref{INTRO:eq:CSLGENERAL}
with the better exponents $\tau_N$ obtained in \eqref{CSL:2}, 
and by removing the John condition \eqref{A5JOHN}, provided that
we make a different choice of $z$. 
%
%
%
%
%
%

The precise statement, that can be deduced from more general results obtained in Theorems~\ref{thm:tubular.relaxed_stability_radii} and~\ref{IUhtkS45}, is the following:

\begin{thm}\label{MAIN-TRE}
Let~$\Om\setminus\overline\om \subset \RR^N$ be a bounded domain satisfying assumptions \eqref{A3bis} and \eqref{A4bis}.
Let $u$ satisfy \eqref{eq:problem}, \eqref{eq:overdetermination}
and~\eqref{assumption:regularityuptobordino}.
Assume that~$u \le 0$ on $\pa \om$.
Set
\begin{equation}\label{LANUOVAzqui}
z:= \frac{1}{|\Om^c_{r_i} |} \left\lbrace \int_{\Om^c_{r_i}} x \, dx - N \int_{\Ga_{r_i}} u \; \nu \, dS_x \right\rbrace  ,
\end{equation}
where~$\Om^c_{r_i}$ denotes the points in~$\Om$ which lie at distance
strictly less than~$r_i$
from~$\Ga$, and~$\Ga_{r_i}$ denotes the points in~$\Om$ which lie at distance~$r_i$
from~$\Ga$.

If $z \in \Om$, then there exist~$\rho_e \ge \rho_i>0$
such that
\begin{equation*} B_{\rho_i}(z)\subset \Omega\subset B_{\rho_e}(z)\end{equation*}
and
\begin{equation*}
\rho_e-\rho_i\le C \,\big({\mathcal{H}}^{N-1}(\pa^* \om)\big)^{\tau_N/2},
\end{equation*}
where
$C>0$ is a constant depending only
on $N$, the internal radius~$r_i$
in~\eqref{A4bis}, $d_\Om$, and $\|u\|_{C^2(\partial\omega)}$. The exponents $\tau_N>0$ depend only
on~$N$.
\end{thm}

We stress that the approach used in Theorem~\ref{MAIN-TRE} does not
need
the assumption in~\eqref{A5JOHN}
that $\Om \setminus \ol{\om}$ is a John domain and hence the dependence on $b_0$, present in \eqref{INTRO:eq:CSLGENERAL}, has been dropped.

Interestingly, the values of the exponents~$\tau_N$ in 
Theorem~\ref{MAIN-TRE} are the same as those in Theorem~\ref{MAIN}
(and therefore they are ``better'' than the ones obtained in \eqref{INTRO:eq:CSLGENERAL},
though they rely on a different choice of~$z$).

We think that it would be interesting to investigate the possible optimality of
these exponents also in the framework provided by Theorem~\ref{MAIN-TRE}.

We remark that the setting of~$z$ in~\eqref{LANUOVAzqui}
is modeled on the annular set~$\Om^c_{r_i}$
rather than on $\Om\setminus\overline\om$ as in \eqref{eq:introstaementez}.
%
%
It would be interesting to further investigate
the impact of different possible choices for $z$.

\subsection{Organization of the paper}

The rest of this paper is organized as follows. 
Section~\ref{S2} contains some detailed motivations
from shape optimization, fluid dynamics and mechanics
which naturally lead to the problem considered in this paper.

In Section~\ref{NOATZ}
we make some notation precise.

In Section~\ref{S:1}, we present some integral identities
of Rellich-Pohozaev-type for solutions of~\eqref{eq:problem}.
In these computations, one does not need to impose the additional
condition in~\eqref{eq:overdetermination} from the beginning,
and aims at comparing a weighted ``deficit'' on~$\Omega\setminus\overline\omega$
(measuring ``how far
from rotational invariant'' the solution is) with suitable
surface integrands on~$\Gamma$ and~$\partial\omega$. {F}rom
these identities, the auxiliary information in~\eqref{eq:overdetermination} provides a more precise, and simpler information.

In Section~\ref{sec:Some estimates} we collect useful estimates and we use them to obtain a suitable
stability bound on the spherical pseudo-distance defined in \eqref{INTRO:eq:pseudodistance}
%
%
(see Theorem \ref{thm:stabestimateforpsudodistance}). We also put in relation this pseudo-distance with the asymmetry defined in \eqref{INTRO:eq:asymmetry} (see Lemma~\ref{lem:relationasymmetrypseudodistance}).

The estimates collected in Section \ref{sec:Some estimates} are then also used in Section~\ref{I:E} to bound the
difference~$\rho_e - \rho_i$.

{I}n Section~\ref{SEC-6}, from the information obtained in Sections \ref{sec:Some estimates} and \ref{I:E}, we obtain a number of quantitative results,
%
%
that also contain Theorem~\ref{MAIN}
as a particular case.

Finally, in Section \ref{sec:relaxing assumptions} we obtain generalizations of the quantitative results presented in Section \ref{SEC-6}, by relaxing the regularity assumptions \eqref{A3} and \eqref{A4},
and hence we establish Theorems~\ref{MAIN-DUE} and~\ref{MAIN-TRE} as particular cases.

\section{Models and motivations}\label{S2}

\subsection{Heating with source malfunctioning}
A natural motivation for problem~\eqref{eq:problem}-\eqref{eq:overdetermination}
comes from the {\em optimal heating theory}. In this setting,
a region~$\Omega$ is given which is in direct
contact with an external environment having constant
(say, zero) temperature. In this setting,
at the equilibrium, the temperature on the boundary of~$\Omega$
is set to be zero. The region~$\Omega$
is also provided by a fine set of heating devices.
All these devices are the same and produce the same heating effect
with the exception of those placed in a small subregion~$\omega$
which are malfunctioning, see Figure~\ref{MALF}.

\begin{figure}
    \centering
    \includegraphics[width=11cm]{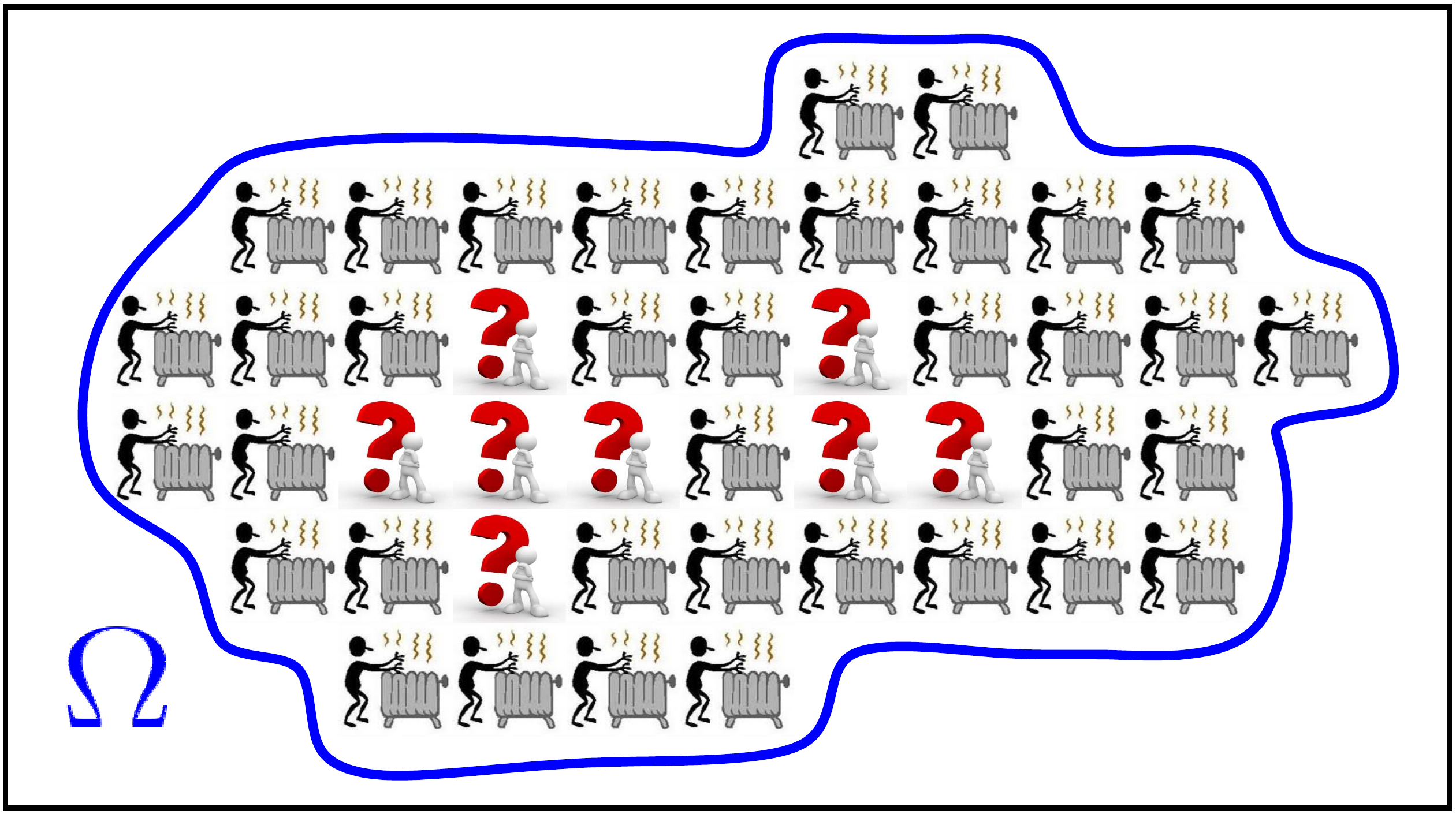}
    \caption{\em {{Matilde's problem: a region with some malfunctioning heating devices.}}}
    \label{MALF}
\end{figure}

In this setting, at the equilibrium,
the local heat flow, normalized by the flow surface, is constant,
say equal to~$1$, in~$\Omega\setminus\omega$,
but it may be different from~$1$ in~$\omega$.
In a nutshell, the mathematical description of this situation is
given by
\begin{equation}\label{0-0-1}
\begin{cases}
\Delta u=f&{\mbox{ in }}\Omega,\\
u=0&{\mbox{ on }}\partial\Omega,\end{cases}
\end{equation}
with~$f(x)=1$ for every~$x\in\Omega\setminus\omega$.

A natural question in this setting is to optimize the shape of~$\Omega$ in order to store in the domain the biggest possible
amount of caloric energy. Concretely, given~$u=u^{(f)}$ satisfying~\eqref{0-0-1},
one may try to select~$\Omega$ in order to maximize (among the domains of fixed measure) the functional
\begin{equation}\label{0-0-2} \int_\Omega |\nabla u^{(f)}(x)|^2\,dx.\end{equation}
This variational problem is well posed
(see e.g. Theorem~4.5.2 in~\cite{HP}) and it is
naturally related to the boundary
derivative prescription
\begin{equation}\label{0-0-3}
{\mbox{$\partial_\nu u^{(f)}={\rm const}$
on~$\partial\Omega$,}}\end{equation} thus leading to the problem described in~\eqref{eq:problem}-\eqref{eq:overdetermination}. We
refer to Proposition 6.1.10 in~\cite{HP}
and Appendix~\ref{APP:he} here for a direct computation relating the
shape optimization of~\eqref{0-0-2}
and the Neumann condition in~\eqref{0-0-3}.

In this framework, our result\footnote{Strictly speaking, \eqref{0-0-1} with a source $f \ge 0$ should be described as
%
an optimal cooling, rather than heating, problem:
speaking of heating problem should require to add a minus sign in front of $\De$ in \eqref{0-0-1}.
Nevertheless, we preferred to keep the sign convention in accordance with \eqref{eq:problem} and the rest of the paper.
%
%
}
in Theorem~\ref{MAIN} states that, for a small region~$\omega$ of
malfunctioning of the heat source, the optimal domain~$\Omega$ is necessarily close to a ball (with a quantitative information on the proximity between~$\Omega$ and a suitable ball).
This result is close to intuition, since jagged domains end up dissipating
most of the caloric energy from their boundaries.

\subsection{Laminar flows with a small tube of unknown density}
Another motivation of the problem in~\eqref{eq:problem}-\eqref{eq:overdetermination} comes from {\em laminar flows}, as modeled by a Navier-Stokes equation of the type
\begin{equation}\label{AJjjd55}
\partial_t(\rho v) + {\rm div}(\rho v\otimes v) -\mu\Delta v + \nabla p= -\rho g,
\end{equation}
where~$v$ is the vectorial velocity of the fluid,
$\rho$ is its density, $\mu$ is its viscosity coefficient,
$p$ is the pressure, $g$ is the gravity acceleration and~$\otimes$
denotes the outer product (i.e., given two vectors~$a$ and~$b$,
$a\otimes b$ is the matrix whose $(i,j)$th entry is~$a_ib_j$)
see e.g.~\cite{Da}.
The incompressibility condition
\begin{equation}\label{AJjjd51}
{\rm div}\,v=0\end{equation}
leads to
$$ {\rm div}(\rho v\otimes v)=(v\cdot\nabla)(\rho v).$$
In this way,
one obtains from~\eqref{AJjjd55} that
\begin{equation}\label{AJjjd57}
\partial_t(\rho v) + (v\cdot\nabla)(\rho v) -\mu\Delta v + \nabla p= -\rho g.
\end{equation}
One assumes that the flow is ``vertical'',
namely~$v=(0,0,u)$ for some scalar function~$u$, thus obtaining 
that
\begin{equation}\label{AJjjd59} (v\cdot\nabla)(\rho v)=((0,0,u)\cdot\nabla)(\rho (0,0,u))=
u\,\partial_3(\rho (0,0,u)).\end{equation}
We also suppose that the laminar flow occurs in a ``vertical tube''
of the type~$\Omega\times\RR$, with~$\Omega\subset\RR^2$,
and that the density of the fluid only depends on the horizontal
position (that is, the fluid maintains the same density along its
vertical flow). In this setting, we have that~$\rho=\rho(x_1,x_2)$
and accordingly, by~\eqref{AJjjd51} and~\eqref{AJjjd59},
$$ (v\cdot\nabla)(\rho v)=(0,0,
\rho u\,\partial_3 u)=(0,0,\rho u\,{\rm div}\,v)=0.$$
Hence, we deduce from the third component of~\eqref{AJjjd57} that
\begin{equation}\label{AJjjd60}
\rho\partial_t u   -\mu\Delta u + \partial_3 p=-\rho g.
\end{equation}
The case of a ``steady state'' flow (that is~$\partial_t u=0$)
with constant pressure (hence~$\partial_3 p= 0$) reduces~\eqref{AJjjd60} to
\begin{equation}\label{AJjjd61}
\Delta u =\frac{\rho g}\mu.
\end{equation}
If the fluid has constant (say, up to changing the inertial
reference frame, zero) velocity at the boundary of the pipe,
equation~\eqref{AJjjd61} is complemented by the boundary condition
\begin{equation}\label{AJjjd62}
u=0\qquad{\mbox{on }}\partial\Omega.\end{equation}
Also, if the fluid presents a constant tangential stress on the pipe,
we have that
\begin{equation}\label{AJjjd63}
\partial_\nu u={\rm constant}\qquad{\mbox{on }}\partial\Omega.\end{equation}
The case described in~\eqref{eq:problem}-\eqref{eq:overdetermination}
is, in this setting, a byproduct of~\eqref{AJjjd61},
\eqref{AJjjd62} and~\eqref{AJjjd63} in which
the density of the fluid (as well as its viscosity and the gravity
acceleration) is constant in the region~$\Omega\setminus \ol{\omega}$,
but it is possibly unknown in~$\omega$.
In this framework, our result in Theorem~\ref{MAIN}
says that
if a laminar fluid is homogeneous out of a small region
and presents a constant tangential stress on the pipe, then necessarily
the pipe is close to a right circular cylinder.

\subsection{Traction of beams with small inhomogeneity}
In the theory of elasticity,
one can consider
the displacement vector~$U=(U_1,U_2,U_3)$
that describes the deformation of some material.
Also, it is customary to introduce the stress tensor
\begin{equation}\label{j3}
\sigma_{ij}:=
\partial_i U_j+\partial_j U_i,\end{equation}
describing the force (per unit area)
in the $i$th direction along the infinitesimal surface
orthogonal to~$e_j$, being~$e_1:=(1,0,0)$, $e_2:=(0,1,0)$
and~$e_3:=(0,0,1)$, see e.g. formula~(129) in~\cite{Hj}
(here, we are supposing the strain and the stress to be proportional,
setting the proportionality constant equal to~$1$ for the sake of
simplicity).

In this framework, the equilibrium configurations
are those for which the forces are infinitesimally balanced, that is
\begin{equation}\label{EQysa}
\sum_{j=1}^3 \partial_j\sigma_{ij}=0\qquad{\mbox{for all }}i\in\{1,2,3\},\end{equation}
see e.g. formula~(189) in~\cite{Hj}.

More specifically,
we focus here on the torsion of a vertical beam~${\mathcal{B}}$
(see e.g.~\cite[pages 100-120]{So}).
The fact that the beam is vertical means, in our setting, that
\begin{equation}\label{BA}
{\mathcal{B}}=\{(x_1,x_2,x_3) {\mbox{ s.t. }}(x_1,x_2)\in\Omega_{x_3}\},\end{equation}
for a suitable (bounded and smooth) family
of domains~$\Omega_{x_3}\subset\RR^2$.
We assume that each point of the beam lying on a given
horizontal plane~$\{x_3=\bar x_3\}$ performs a horizontal
rotation of a small angle~$\vartheta=\vartheta(x)$,
plus a vertical movement that is the same for every
height of the bar (we will denote this vertical movement
by~$v=v(x_1,x_2)$).
In this setting, the torsion of the bar moves a point~$x=(x_1,x_2,\bar x_3)$
to the point
$$X:=(x_1\cos\vartheta+x_2\sin\vartheta,
-x_1\sin\vartheta+x_2\cos\vartheta,\bar x_3+v)\simeq
(x_1+x_2\vartheta,-x_1\vartheta+x_2,\bar x_3+v),$$ and so,
in this approximation, we can write the displacement vector as
\begin{equation*}
U=X-x=( \vartheta\,x_2,-\vartheta\,x_1,v).\end{equation*}
{F}rom this and~\eqref{j3}, we can write
\begin{equation}\label{HAS34689JSA}
\begin{split}&
\sigma_{31}=\partial_3 U_1+\partial_1 U_3=\partial_3 \vartheta \,x_2+
\partial_1 v,
\\&
\sigma_{32}=\partial_3 U_2+\partial_2 U_3
=-\partial_3 \vartheta\,x_1+\partial_2 v,\\
{\mbox{and}}\qquad
&\sigma_{33}=2\partial_3 U_3=0.\end{split}
\end{equation}
As a result, exploiting the latter equation and~\eqref{EQysa} with~$i:=3$,
\begin{equation}\label{A83uf565}
0=\sum_{j=1}^3 \partial_j\sigma_{3j}=
\partial_1\sigma_{31}+\partial_2\sigma_{32}
.\end{equation}
We then 
fix~$x_3=1$ and
consider the $1$-form on~$\RR^2$
given by
\begin{equation}\label{A83uf565ddd} \alpha:=-\sigma_{32}\,dx_1+\sigma_{31}\,dx_2,\end{equation}
and we deduce from~\eqref{A83uf565} that
$$ d\alpha=\big(
\partial_2\sigma_{32}+\partial_1\sigma_{31}
\big)\,dx_1\wedge dx_2=0.$$
This gives that there exists a ``warping potential'' $u$ such that~$\alpha=du$ as $1$-forms in~$\RR^2$. Accordingly, by~\eqref{A83uf565ddd},
we have that
\begin{equation}\label{083ru776764848} \partial_1u=-\sigma_{32}\qquad{\mbox{and}}\qquad
\partial_2u=\sigma_{31}.
\end{equation}
Combining this with~\eqref{HAS34689JSA}, one sees that
$$ \partial_1u=\partial_3 \vartheta\,x_1-\partial_2 v\qquad{\mbox{and}}\qquad
\partial_2u=\partial_3 \vartheta \,x_2+
\partial_1 v,
$$
and therefore
\begin{equation}\label{THE01}
\Delta u= \partial_1\big(\partial_3 \vartheta\,x_1-\partial_2 v\big)+
\partial_2\big(
\partial_3 \vartheta \,x_2+
\partial_1 v\big)=2\partial_3\vartheta+\partial^2_{1,3}\vartheta\,x_1+
\partial^2_{2,3}\vartheta\,x_2,
\end{equation}
with these functions evaluated at~$x_3=1$.
As a special case, one can take into account the situation in which
the angle~$\vartheta$ depends linearly on the height of the beam,
say~$\vartheta(x)=\theta(x_1,x_2)\,x_3$
(this is physically reasonable, for instance, if the beam is constrained
at~$\{x_3=0\}$ and some torque is applied from the top of the beam). In this setting,
\eqref{THE01} reduces to
\begin{equation}\label{THE02-0}
\Delta u=2\theta+\partial_{1}\theta\,x_1+
\partial_2\theta\,x_2.
\end{equation}
If we suppose that the beam is built by two different materials,
one occupying~$\Omega\setminus \ol{\omega}$
and the other
occupying~$\omega$
(where~$\Omega$ here represents the domain~$\Omega_{x_3}$
in~\eqref{BA} with~$x_3=1$), the expression in~\eqref{THE02-0}
takes two different forms in~$\Omega\setminus \ol{\omega}$
and~$\omega$. In particular, if we know that the material in~$\Omega\setminus \ol{\omega}$ is homogeneous
we can suppose that the horizontal
rotation is uniform there and thus~$\theta$ is independent on the point (i.e., $\partial_1\theta=
\partial_2\theta=0$), hence deducing from~\eqref{THE02-0} that
\begin{equation}\label{THE02}
\Delta u={\rm constant}\qquad{\mbox{in }}\Omega\setminus \ol{\omega} .
\end{equation}
Also, the surface traction,
as a force per
unit of area, at a boundary point of the beam is defined
as the normal component of the vertical stress, that is
$$ T:=(\sigma_{31},\sigma_{32},\sigma_{33})\cdot \bar\nu,$$
being~$\bar\nu\in\RR^3$ the normal to the beam (see e.g. 
the third component in formula~(174) of~\cite{Hj}).
Recalling~\eqref{HAS34689JSA}, we have that
$$ T=(\sigma_{31},\sigma_{32},0)\cdot \bar\nu=
(\sigma_{31},\sigma_{32})\cdot\nu,$$
being~$\nu\in\RR^2$ normal to~$\Omega$.
Thus, in view of~\eqref{083ru776764848},
$$ T=(\partial_2u,-\partial_1u)\cdot\nu=
\nabla u\cdot\tau,
$$
being~$\tau:=(-\nu_2,\nu_1)$ a unit tangent vector to~$\Omega$.
Therefore, if the traction vanishes, the tangential derivative
of~$u$ vanishes as well, hence~$u$ is constant along~$\partial\Omega$.
Since~$u$ was introduced as a potential, it is defined up to an additive
constant, hence we can rephrase these considerations by saying that
if the traction vanishes, then
\begin{equation}\label{THE03}
u=0\qquad{\mbox{on }}\partial\Omega.
\end{equation}
We also remark that, if~$\sigma_3=(\sigma_{31},\sigma_{32},\sigma_{33})$, then
$$|\sigma_3|=|(\sigma_{31},\sigma_{32},0)|=|\nabla u|,$$
thanks to~\eqref{HAS34689JSA} and~\eqref{083ru776764848}.

Hence, for constant stress intensity~$|\sigma_3|$, we deduce from~\eqref{THE03} that
\begin{equation}\label{THE04}
\partial_\nu u={\mbox{constant}}\qquad{\mbox{on }}\partial\Omega.
\end{equation}
We thus
observe that problem~\eqref{eq:problem}-\eqref{eq:overdetermination}
arises naturally from~\eqref{THE02}, \eqref{THE03}
and~\eqref{THE04}, and,
in this framework, Theorem~\ref{MAIN} says that
if a beam is homogeneous out of a small region
and presents zero traction and constant stress intensity, then necessarily the horizontal
section of the beam is close to a disk.

\section{Notation}\label{NOATZ}

Unless differently specified,
we will denote by $\Om\subset\RR^N$, with $N\ge 2$,
a connected, bounded open set, and call $\Ga:=\partial\Omega$ its boundary. 
We will denote indifferently by $|\Om|$ and $|\Ga|$
the $N$-dimensional Lebesgue measure of $\Om$
and the surface measure of $\Ga$. 

Moreover, we will denote by $d_\Om$ the diameter of $\Om$, that is
\begin{equation}\label{DIAMET} d_\Om:=\sup_{x,y\in\Om}|x-y|.\end{equation}
Also, we shall denote by $\om $ an open subset of $\Om$, such that~$\overline\om\subset\Om$.

For all~$x \in \Om \setminus \ol{\om}$, we consider the distance function
\begin{equation}\label{DIST}
\de (x) := \dist(x, \Ga \cup \pa \om).
\end{equation}
\smallskip

As already mentioned in the Introduction, we first consider the
case in which, being \eqref{A3} in force and by recalling \eqref{noWea}, $\Om \setminus \ol{\om}$ is of class~$C^1$:
in this setting~$\nu$ denotes the (exterior) unit normal vector field
to~$\Om \setminus \ol{\om}$.
Then, we will clarify in Section~\ref{sec:relaxing assumptions}
the notation that we use when the regularity assumption in~\eqref{A3} is dropped.
\smallskip

Now, we clarify the notation used for the spaces~$C^k$ and we discuss
the regularity assumption in~\eqref{assumption:regularityuptobordino}.

As usual, for an open set $D \subset \RR^N$, and $k$
a positive integer, $C^k (D)$ denotes the space of functions possessing continuous derivatives up to order $k$ on $D$.

By $C^k(\ol{D})$ we denote the space of functions that are
restrictions to~$\ol{D}$ of functions in~$C^k (\RR^N)$.
In other words, $C^k (\ol{D})$ is the space of the functions in $C^k (D)$
that can be extended to $C^k (\RR^N)$. The same definition has been adopted
for instance in
\cite[Appendix C, pag. 562]{Leoni} and also in many books
of Differential Geometry.

We recall that, thanks to Whitney extension theorem
(see the original paper \cite{Wh} or \cite[Theorem 2.3.6]{KP-Primer}),
this definition is equivalent to \cite[Definition 2.3.5]{KP-Primer} based on
a Taylor-expansion condition.

Moreover, if $D$ satisfies property (P) of \cite{Wh-boundaries} -- i.e.,
quasiconvexity, that in particular is surely satisfied if $D$ is Lipschitz
(see \cite[Sections 2.5.1, 2.5.2]{Br}) -- then, thanks to the main theorem
on page~485
in \cite{Wh-boundaries}, our definition of $C^k(\ol{D})$
(as well as \cite[Definition 2.3.5]{KP-Primer})
is also equivalent to the definition used in many books of PDEs
(e.g., \cite{GT}), that is:
$C^k ( \ol{D})$ is the space of functions in $C^k (D)$
whose derivatives up to order~$k$ have continuous extensions to the
closure~$\ol{D}$.

For more general domains, \cite[Definition 2.3.5]{KP-Primer} (as well
as our definition) is stronger than that adopted in \cite{GT}. 
For more details on this subject we refer to \cite[Section~2.3]{KP-Primer},
\cite[Section~5.2]{KP}, \cite{Kra}, \cite{Br}, and \cite[Appendix C]{Leoni}.

In our setting, by taking $D:= \Om \setminus \ol{\om}$ and $k:=2$
in the definition of $C^k (\ol{D})$, we have that the
assumption in~\eqref{assumption:regularityuptobordino} guarantees
that~$u$ can be extended to a~$C^2$ function throughout~$\RR^N$. 
This will allow us to perform the generalizations described in
Section~\ref{sec:relaxing assumptions} when the regularity assumption
on $\pa \om$ in~\eqref{A3} is dropped and replaced just by \eqref{A3bis}. 
In particular, when integrating by parts, we can still write the derivatives of $u$ up to the second order on the reduced boundary $\pa^* \om$.

We remark that up to Section~\ref{SEC-6},
we will take~$\Om \setminus \ol{\om}$ to be of class $C^1$,
and therefore assumption \eqref{assumption:regularityuptobordino} has
univocal meaning, no matter what definition of $C^2( \ol{\Om} \setminus \om )$
we adopt (among the three presented above). 

\section{Integral identities}\label{S:1}

The goal of this section is to develop a series of integral identities which will be conveniently exploited
to deduce quantitative bounds on the solution of~\eqref{eq:problem}-\eqref{eq:overdetermination}
and on its domain of definition. 
We start by proving a Rellich-Pohozaev-type identity
and its consequences. Notice that in the next two statements we are not imposing yet
the overdetermined condition in~\eqref{eq:overdetermination}.

\begin{lem}[A Rellich-Pohozaev-type identity]
\label{lem:AReillichPohozaevIdentity}
Suppose that~$\Om\setminus\overline\om$
is of class~$C^1$.
If $u \in C^2 (\ol{\Om} \setminus \om)$ satisfies \eqref{eq:problem}, then the following identity holds:
\begin{equation}\begin{split}\label{eq:Pohozaevadhoc}&
(N+2) \int_{ \Om \setminus \ol{\om} }  | \na u|^2  dx =  \int_\Ga <x , \nu> u_{\nu}^2 \, dS_x 
\\&\qquad\qquad+
 2 N \int_{\pa \om} \left\lbrace u \, u_\nu - \frac{< x , \nu>}{N} u + \frac{ <x, \na u > }{N} u_\nu - \frac{ <x , \nu> }{2N} | \na u |^2 \right\rbrace  dS_x .
\end{split}\end{equation}
\end{lem}
\begin{proof}
By a direct computation, it is easy to verify the following differential identity (valid for every function~$u$):
\begin{equation}\label{OR}
\dv \left\lbrace <x, \na u> \na u -   \frac{| \na u|^2}{2} \, x \right\rbrace = <x,\na u> \De u - \frac{N-2}{2} \, |\na u|^2 .
\end{equation}
We now specialize this identity to a solution of~\eqref{eq:problem}.
For this, we remark that the Dirichlet boundary condition in~\eqref{eq:problem}
gives that
\begin{equation}\label{OR2}
{\mbox{$\na u = u_\nu \, \nu$ on $\Ga$.}}\end{equation}
We also remark that
\begin{equation}\label{OR3}
\partial\big({\Om \setminus \ol{\om} }\big)=\Gamma\cup\partial\omega.
\end{equation}
Hence, by integrating~\eqref{OR} over $\Om \setminus \ol{\om}$,
exploiting~\eqref{eq:problem}, \eqref{OR2}
and~\eqref{OR3}, and applying the
divergence theorem, we see that
\begin{equation}\begin{split}\label{eq:1perPohozaev}&
\int_{\Om \setminus \ol{\om} } \left\lbrace <x,\na u> - \frac{N-2}{2} \, |\na u|^2 \right\rbrace \, dx 
\\=\;&
\int_{\Om \setminus \ol{\om} } \left\lbrace <x,\na u>\Delta u - \frac{N-2}{2} \, |\na u|^2 \right\rbrace \, dx 
\\=\;&
\int_{\Om \setminus \ol{\om} }
\dv \left\lbrace <x, \na u> \na u -   \frac{| \na u|^2}{2} \, x \right\rbrace
\,dx
\\=\;&
\int_{\partial(\Om \setminus \ol{\om}) }
\left\lbrace <x, \na u> \;< \na u , \nu  >-   \frac{| \na u|^2}{2} \, <x,\nu> \right\rbrace
\,dS_x
\\ =\;&
\frac{1}{2} \int_{\Ga} <x, \nu> u_\nu^2 \, dS_x + \int_{\pa \om} \left\lbrace <x, \na u> \, u_\nu -  | \na u|^2 \frac{ <x, \nu > }{2}  \right\rbrace \, dS_x. 
\end{split}\end{equation}
Now we observe that, in $\Om \setminus \ol{\om}$,
\begin{eqnarray*}
\dv{ \left( \frac{x}{N} \, u - u \, \na u \right) } &=&
u+\frac{< x , \na u >}{N}-| \na u|^2-u\,\Delta u
\\&=& \frac{< x , \na u >}{N} - | \na u|^2 ,
\end{eqnarray*}
where the equation in~\eqref{eq:problem} has been used in the last step.

As a consequence,
$$ < x , \na u >=N\,| \na u|^2+N\,\dv{ \left(\frac{ x}N \, u - u \, \na u \right) } .$$
By integrating this identity over $\Om \setminus \ol{\om}$,
and using the boundary condition in~\eqref{eq:problem}, we deduce that
$$
\int_{\Om \setminus \ol{\om} } < x, \na u  > \, dx = N \int_{\Om \setminus \ol{\om} } | \na u|^2 \, dx + N \int_{ \pa \om } \left\lbrace \frac{ <x, \nu > }{N} u - u \, u_\nu \right\rbrace dS_x .
$$
By putting together the last identity and \eqref{eq:1perPohozaev}, 
we obtain~\eqref{eq:Pohozaevadhoc}, as desired.
\end{proof}

We now obtain another useful
integral identity, which is based on~\eqref{eq:Pohozaevadhoc}
and a suitable $P$-function computation.

\begin{thm}
\label{thm:IdentityGeneral}
Suppose that~$\Om\setminus\overline\om$
is of class~$C^1$.
If $u \in C^2 (\ol{\Om} \setminus \om)$ satisfies \eqref{eq:problem}, then the following identity holds:
\begin{equation}\label{eq:FId}\begin{split}&
\int_{\Om \setminus \ol{\om} } (-u) 2 \left\lbrace | \na^2 u|^2- \frac{(\De u)^2}{N} \right\rbrace dx 
\\ =\;&
\int_\Ga u_\nu^2 \left( u_\nu - \frac{ < x, \nu>}{N} \right) \, dS_x +
 \int_{\pa \om } 2u \left( \frac{< x, \nu >}{N} - u_\nu \right) dS_x 
\\ &\quad+
\int_{\pa \om } \left\lbrace u_\nu | \na u |^2 - 2 \frac{< x, \na u >}{N} u_\nu + | \na u|^2 \frac{< x , \nu >}{N}
+ \frac{2}{N} u u_\nu - 2 <\na^2 u \na u , \nu > u \right\rbrace dS_x .
\end{split}\end{equation}
\end{thm}

We point out that
the term in the braces in the left-hand side of \eqref{eq:FId} could be written as~$
\left\lbrace | \na^2 u|^2- \frac{1}{N} \right\rbrace $.
Nevertheless, we preferred to use the
notation
\begin{equation}\label{eq:CSRIMESSOALLAFINEforseceragia}
 | \na^2 u|^2 - \frac{(\De u)^2}{N} 
\end{equation}
to emphasize that this quantity plays the role of a
{\it Cauchy-Schwarz deficit}. In fact, by Cauchy-Schwarz inequality we have that \eqref{eq:CSRIMESSOALLAFINEforseceragia} is nonnegative, and equals $0$ in $\Om \setminus \ol{\om}$ if and only if $\Om$ is a ball of radius $R=N |\Om| / |\Ga|$ and $u(x)=\frac{|x|^2 - R^2}{2N} $ in $\Om \setminus \ol{\om}$, up to translations (see, e.g., \cite[Lemma 1.9]{Pog}).

\begin{proof}[Proof of Theorem \ref{thm:IdentityGeneral}]
If we set
\begin{equation}\label{eq:defP}
P:= | \na u |^2 - \frac{2}{N} u ,
\end{equation}
a direct computation, valid for any function, informs us that
\begin{equation}\label{P:Ak} \Delta P = 2 | \na^2 u|^2+2<\nabla u,\nabla\Delta u>- \frac{2}{N}\Delta u,\end{equation}
where~$\na^2 u$ denotes the Hessian matrix of~$u$
and~$| \na^2 u|$ denotes its Frobenius norm, that is
$$ (\na^2 u)_{i,j}= \partial^2_{ij}u\qquad
{\mbox{and}}\qquad| \na^2 u|^2=\sum_{i,j=1}^n(\partial^2_{ij}u)^2.$$
Then, using~\eqref{P:Ak}
and the equation in~\eqref{eq:problem},
we conclude that
\begin{equation}\label{eq:CSdeficit}
\De P = 2 \left\lbrace | \na^2 u|^2- \frac{(\De u)^2}{N} \right\rbrace  \quad \mbox{in } \Om \setminus \ol{\om} .
\end{equation}
On the other hand, by~\eqref{eq:problem} and the Green identity,
\begin{equation}\label{eq:1perdim_identity}
\int_{\Om \setminus \ol{\om} }  (-u) \De P \, dx = - \int_{ \Om \setminus \ol{\om} } P \, dx + \int_{\Ga} u_\nu P \, dS_x + \int_{\pa \om} \left\lbrace u_\nu P - u \, P_\nu \right\rbrace dS_x .
\end{equation}
Let us work on the first integral on the right-hand side of \eqref{eq:1perdim_identity}.
For this,
integrating over $ \Om \setminus \ol{\om} $ the differential identity
$$
\dv(u \na u)= u \De u + | \na u|^2 ,
$$
and recalling \eqref{eq:problem}, we get that
$$
\int_{ \Om \setminus \ol{\om} } u \, dx=
\int_{ \Om \setminus \ol{\om} } u\De u \, dx
 = - \int_{ \Om \setminus \ol{\om} } | \na u |^2 \, dx + \int_{ \pa \om } u \, u_\nu \, dS_x .
$$
Thus, in light of~\eqref{eq:defP},
\begin{eqnarray*}
- \int_{ \Om \setminus \ol{\om} } P \, dx& =& 
\int_{ \Om \setminus \ol{\om} } \left\lbrace \frac{2}{N} u - | \na u |^2 \right\rbrace \, dx   
\\&=&
- \frac{N+2}{N} \int_{ \Om \setminus \ol{\om} } | \na u |^2 \, dx + \frac{2}{N} \int_{ \pa \om } u \, u_\nu \, dS_x .
\end{eqnarray*}
Consequently, using \eqref{eq:Pohozaevadhoc},
\begin{equation}\label{eq:3perdim_identity_FIRSTTERM}
\begin{split}
- \int_{ \Om \setminus \ol{\om} } P \, dx \,=\,& - \int_\Ga \frac{ <x , \nu> }{N} u_{\nu}^2 \, dS_x 
\\&\quad-
 2  \int_{\pa \om} \left\lbrace u \, u_\nu - \frac{< x , \nu>}{N} u + 
\frac{ <x, \na u > }{N} u_\nu - \frac{ <x , \nu> }{2N} | \na u |^2 \right\rbrace  dS_x 
\\&\quad+\frac{2}{N} \int_{ \pa \om } u \, u_\nu \, dS_x  .
\end{split}
\end{equation}

Moreover, to deal with the second integral in the right-hand side of \eqref{eq:1perdim_identity}, by recalling~\eqref{eq:problem} and~\eqref{eq:defP}, we have
\begin{equation}\label{eq:4perdim_identity_SECONDTERM}
\int_{\Ga} u_\nu P \, dS_x = \int_{\Ga} u_\nu^3 \, dS_x .
\end{equation}

Furthermore, using \eqref{eq:defP}, we see that
\begin{eqnarray*}
P_\nu = 2 <\na^2 u\na u, \nu> - \frac{2}{N} u_\nu
\end{eqnarray*}
and accordingly
\begin{eqnarray*}
u_\nu P - u \, P_\nu& =&
u_\nu\left\lbrace
| \na u |^2 - \frac{2}{N} u
\right\rbrace
-2u <\na^2 u\na u, \nu> + \frac{2}{N} uu_\nu\\
&=&u_\nu | \na u|^2 - 2 u \, < \na^2 u \na u,\nu>.
\end{eqnarray*}
As a result,
the third integral in the right-hand side of \eqref{eq:1perdim_identity} becomes
\begin{equation}\label{eq:5perdim_identity_THIRDTERM}
\int_{\pa \om} \left\lbrace u_\nu P - u \, P_\nu \right\rbrace dS_x =
\int_{\pa \om} \left\lbrace u_\nu | \na u|^2 - 2 u \, < \na^2 u \na u,\nu> \right\rbrace dS_x .
\end{equation}

Thus, \eqref{eq:FId} follows by putting together~\eqref{eq:CSdeficit}, \eqref{eq:1perdim_identity}, \eqref{eq:3perdim_identity_FIRSTTERM}, \eqref{eq:4perdim_identity_SECONDTERM} and~\eqref{eq:5perdim_identity_THIRDTERM}.
\end{proof}

We now impose the overdetermined condition \eqref{eq:overdetermination},
and we obtain from~\eqref{eq:FId}
the following integral identity:

\begin{cor}
\label{cor:Identitydoposovradeterminazione}
Suppose that~$\Om\setminus\overline\om$
is of class~$C^1$.
If $u \in C^2 (\ol{\Om} \setminus \om)$ satisfies \eqref{eq:problem} and \eqref{eq:overdetermination}, then the following identity holds:
%
%
%
\begin{multline}\label{eq:FIdconsovradeterminazione}
\int_{\Om \setminus \ol{\om} } (-u) 2 \left\lbrace | \na^2 u|^2- \frac{(\De u)^2}{N} \right\rbrace dx 
\\ 
=
c^2 \int_{ \pa \om}   \left( \frac{ < x, \nu>}{N} - u_\nu \right)  dS_x +
\int_{ \pa \om}  2u \left( \frac{< x, \nu >}{N} - u_\nu \right)dS_x
\\
+
\int_{ \pa \om} \left\lbrace u_\nu | \na u |^2 - 2 \frac{< x, \na u >}{N} u_\nu + | \na u|^2 \frac{< x , \nu >}{N} + \frac{2}{N} u u_\nu - 2 <\na^2 u \na u , \nu > u \right\rbrace dS_x .
\end{multline}
\end{cor}

\begin{proof}
We observe that the constant~$c$ in~\eqref{eq:overdetermination}
can be determined explicitly in terms of~$|\Gamma|$,
$|\Omega|$, $|\omega|$ and the values of~$u_\nu$ along~$\partial\omega$.
Indeed, by using \eqref{eq:problem} and \eqref{eq:overdetermination}
together with the divergence theorem,
we deduce that
\begin{equation}\label{eq:value_c}
c |\Ga| = \int_\Ga u_\nu \, dS_x = \int_{ \Om \setminus \ol{\om} } \De u \, dx - \int_{ \pa \om} u_\nu \, dS_x = | \Om| - | \om| - \int_{ \pa \om} u_\nu \, dS_x.
\end{equation}
In particular, we will use that
\begin{equation}\label{eq:value_c2}
\int_\Ga u_\nu \, dS_x = | \Om| - | \om| - \int_{ \pa \om} u_\nu \, dS_x.
\end{equation}
On the other hand, by applying again
the divergence theorem,
\begin{eqnarray*}
| \Om | - | \om |&=&\int_{\Om\setminus\om} 1\,dx\\
&=& \int_{\Om\setminus\om}\frac{\dv x}N\,dx\\
&=& \int_{ \Ga} \frac{<x,\nu>}{N} \, dS_x+
\int_{\pa \om} \frac{ <x, \nu> }{N} \, dS_x.
\end{eqnarray*}
{F}rom this and~\eqref{eq:value_c2}, we conclude that
\begin{eqnarray*}&&
\int_\Ga u_\nu^2 \left( u_\nu - \frac{ < x, \nu>}{N} \right) \, dS_x
=c^2\int_\Ga  \left\{ u_\nu - \frac{ < x, \nu>}{N} \right\}\, dS_x\\
&&\qquad=
c^2\left(
| \Om| - | \om| - \int_{ \pa \om} u_\nu \, dS_x-
\int_\Ga \frac{ < x, \nu>}{N} \, dS_x
\right)\\&&\qquad=
c^2\left(- \int_{ \pa \om} u_\nu \, dS_x+ \int_{\pa \om} \frac{ <x, \nu> }{N} \, dS_x\right).
\end{eqnarray*}
Plugging this information into~\eqref{eq:FId} we obtain the desired result
in~\eqref{eq:FIdconsovradeterminazione}.
\end{proof}

\section{Some estimates on a spherical pseudo-distance}\label{sec:Some estimates}

In this section, we will obtain a suitable bound on the
following pseudo-distance
\begin{equation}\label{eq:pseudodistance}
\int_{\Ga} \left| \frac{|x-z|}{N} - c \right|^2 dS_x ,
\end{equation}
for a suitable~$z\in \RR^N$.

We point out that the quantity in~\eqref{eq:pseudodistance}
plays the role of an ``integral distance'' of $\Ga$ from the
sphere centered at a point~$z\in\Omega$ of radius $Nc$:
indeed,
when~$\Ga=\partial B_{Nc}(z)$, the quantity in~\eqref{eq:pseudodistance}
vanishes, and, in general, this quantity can be considered an~$L^2$-measure
on how far points on~$\Gamma$ are from points on~$\partial B_{Nc}(z)$.
\medskip

We also notice that the pseudo-distance in \eqref{eq:pseudodistance} can be put in relation with the following {\it asymmetry}:
\begin{equation}\label{eq:asymmetryallaFraenkel}
\frac{| \Om \De B_{Nc}(z) |}{| B_{Nc}(z) |},
\end{equation}
where $\Om \De B_{Nc}(z)$ denotes the symmetric difference of $\Om$ and the ball $B_{Nc}(z)$ of radius $Nc$ centered at $z$.
%
%

In particular, the asymmetry in \eqref{eq:asymmetryallaFraenkel} is bounded from above by the pseudo-distance in~\eqref{eq:pseudodistance}, as stated in the following
result:

\begin{lem}\label{lem:relationasymmetrypseudodistance}
Let $\Om \subset \RR^N$ be a bounded domain with Lipschitz boundary $\Ga$, satisfying the uniform interior sphere condition with radius $r_i$. Then, there exists a positive constant $C$, only depending on $N$, $r_i$ and~$c$, such that
\begin{equation*}
\frac{| \Om \De B_{Nc}(z) |}{| B_{Nc}(z) |} \le C \left[ \int_{\Ga} \left| \frac{|x-z|}{N} - c \right|^2 dS_x \right]^{\frac{1}{2}} .
\end{equation*}
\end{lem}

\begin{proof}
The desired result follows by applying \cite[Lemma 11]{Fe} with 
$$K:=\max \left\lbrace \frac{Nc}{r_i} , \, \left( \frac{d_\Om}{2 Nc} \right)^N \right\rbrace \quad \text{ and } \quad r:=Nc .$$ 
Notice that \cite[Lemma 11]{Fe} can be applied with these choices for $K$ and $r$ because the following two relations hold true:
the first is
\begin{equation}\label{eq:stima_asimmetria_volume_riserve}
K |B_{Nc}| \ge \left( \frac{d_\Om}{2 Nc} \right)^N |B_1| (Nc)^N \ge | \Om |,
\end{equation}
where in the last inequality we used that $|B_1| \left( \frac{d_\Om}{2} \right)^N \ge | \Om |$;
the second is
\begin{equation}\label{5.3bis}
K r_{in}(\Om) \ge Nc \, \frac{r_{in} (\Om)}{r_i} \ge Nc ,
\end{equation}
where $r_{in} (\Om):= \max_{\ol{\Om}} \de_\Ga (x)$ denotes the {\it inradius} of $\Om$ and in the last inequality we used that, by definition, $r_{in} (\Om) \ge r_i$.
\end{proof}

To obtain our bounds for the pseudo-distance introduced in~\eqref{eq:pseudodistance}, we recall the notation in~\eqref{DIST}
and we detect an optimal growth of the solution from
the boundary, by adapting an idea from~\cite[Lemma 3.1]{MP2}:

\begin{lem}
\label{lem:relationdist}
Let $u$ satisfy~\eqref{eq:problem}. Assume that~$u \le 0$ on $\pa \om$.
Then,
\begin{equation}\label{instgr}
-u(x) \ge 
\frac{1}{2 N }\,\de(x)^2 
\quad \mbox{ for every } \ x \in \Om \setminus \ol{\om}.
\end{equation}

Moreover, if $\Om \setminus \ol{\om}$ is of class~$C^1$
and satisfies the uniform interior sphere condition with radius $r_i$, that is \eqref{A4}, then it holds that
\begin{equation}
\label{eq:relationdist}
-u(x) \ge \frac{r_i}{2 N}\,\de (x)  \quad \mbox{ for every } \ x \in  \Om \setminus \ol{\om} .
\end{equation}
\end{lem}

\begin{proof}
Let~$x \in \Om\setminus \ol{\om}$ and set~$r:=\de (x)$.
We consider $$
w(y):=\frac{|y-x|^2-r^2}{2 N}.$$
We remark that~$w$ is the solution of the classical torsion problem in $B_r(x)$, namely
\begin{equation}\label{eq:torsionball}
\begin{cases}\De w = 1 \qquad\text{  in } B_r(x), \\ w =0 \qquad\text{  on } \pa B_r(x).\end{cases}
\end{equation}
By comparison we have that $w\ge u$ on $\ol{B}_r(x)$.
In particular, 
$$
- \frac{1}{2 N}\,\de(x)^2 = w(x) \ge u(x) ,
$$ 
and \eqref{instgr} follows.

We point out that~\eqref{eq:relationdist} follows
from~\eqref{instgr} if~$\de (x)\ge r_i$.
Hence, from now on, we can suppose that
\begin{equation}\label{eq62347on}\de (x) < r_i.\end{equation}
Let $\bar x$ be the closest point in $\Ga \cup \pa \om $ to $x$ and call $\tilde B \subset \Om \setminus \ol{\om}$
the ball of radius $r_i$ touching~$\Ga \cup \pa \om $ at~$\bar x$ and containing $x$.
Up to a translation, we can always suppose that 
\begin{equation}\label{8686s2}
{\mbox{the center of the ball $\tilde B$ is the origin.}}\end{equation}
%
%
Now, we let~$\tilde w$ 
be the solution of \eqref{eq:torsionball} in $\tilde B$, that is $\tilde w(y):=\left(|y|^2- r_i^2 \right)/(2N)$.
By comparison, we have that $w \ge u$ in~$\tilde B$, and hence, being $x\in \tilde B$,
\begin{equation}\label{817928y3r}
-u(x) \ge \frac{1}{2 N}\,(r_i^2 - |x|^2 )=
\frac{1}{2 N}\,( r_i + |x| )(r_i-|x|)\ge\frac{ r_i }{2 N} \,(r_i-|x|).\end{equation}
Moreover, from~\eqref{8686s2},
$$r_i-|x|=\de (x).$$
This and~\eqref{817928y3r}
give~\eqref{eq:relationdist}, as desired.
\end{proof}

We recall now some
Hardy-Poincar\'{e}-type inequalities that have been proved in~\cite[Section 3.2]{Pog2} by exploiting the works of
Hurri-Syrj\"anen~\cites{Hu, HS}. 
%
%
In what follows, for a set~$D$ and a function~$v: D \to \RR$, $v_D$ denotes the mean value of $v$ in $D$, that is
\begin{equation}\label{media}
v_D:= \frac{1}{|D|} \, \int_D v \, dx.
\end{equation}
Also, for a function $v:D \to \RR$ we define by~$\nr v \nr_{p, D}$
its $L^p$-norm in~$D$, that is
\begin{equation}\label{ellep}
\nr v \nr_{p, D}:=\left(\int_D|v(x)|^p\,dx \right)^{1/p},
\end{equation}
and
$$
\nr \de^\al \, \na v \nr_{p, D } := \left( \sum_{i=1}^N \nr \de^\al \,  v_i \nr_{p, D }^p \right)^\frac{1}{p} \quad \mbox{and} \quad
\nr \de^\al \, \na^2 v \nr_{p,D} := \left( \sum_{i,j=1}^N \nr \de^\al \, v_{ij} \nr_{p, D}^p \right)^\frac{1}{p},
$$
for $0 \le \al \le 1$ and $p \in [1, \infty)$. Here and whenever no confusion is possible, we will use the abbreviated notation
$$
\de (x) = \dist(x, \pa D) ,
$$
that agrees with \eqref{DIST} when $D= \Om \setminus \ol{ \om }$. 

\begin{lem}\label{lem:John-two-inequalities}
Let $D \subset \RR^N$ be a bounded domain satisfying the uniform interior
sphere condition with radius $r_i$, and consider three real
numbers $r$, $p$ and~$ \al$ such that either
\begin{equation}\label{eq:conditionHS}
1 \le p \le r \le \frac{Np}{N-p(1 - \al )} , \qquad p(1 - \al)<N \quad{\mbox{ and }}
\quad 0 \le \al \le 1 ,
\end{equation}
or
\begin{equation}\label{eq:conditionBS}
r = p \in \left[ 1, \infty \right)  \quad{\mbox{ and }}\quad \al=0.
\end{equation}
Then,

(i) given~$x_0 \in D$, there exists a positive constant $ \mu_{r,p, \al} ( D , x_0 ) $
such that
\begin{equation}
\label{John-harmonic-quasi-poincare}
\nr v \nr_{r, D } \le \mu_{r, p, \al} ( D , x_0 )^{-1} \nr \de^{\al} \, \na v  \nr_{p, D },
\end{equation}
for every function $v$ which is harmonic in $D$ and such that $v(x_0)=0$;
\par
(ii) there exists a positive constant $\ol{\mu}_{r, p, \al} (D)$ such that
\begin{equation}
\label{John-harmonic-poincare}
\nr v - v_{D} \nr_{r,D} \le \ol{\mu}_{r, p, \al} (D)^{-1} \nr \de^{\al} \, \na v  \nr_{p, D },
\end{equation}
for every function $v$ which is harmonic in $D$.

Furthermore, the following explicit bounds hold.
Recalling the notation in~\eqref{DIAMET},
when $r$, $p$ and~$ \al$ are as in \eqref{eq:conditionHS}, we have that
\begin{equation}\label{eq:estimatePoincaremean}
\ol{\mu}_{r, p, \al} (D)^{-1} \le k_{N,\, r, \, p,\, \al} \, \left( \frac{d_D}{r_i} \right)^N |D|^{\frac{1-\al}{N} +\frac{1}{r} +\frac{1}{p} }
\end{equation}
and
\begin{equation}\label{eq:estimatePoincarex_0}
\mu_{r,p,\al}(D, x_0 )^{-1} \le k_{N,r,p,\al} \, \left( \frac{d_D}{\min[r_i, 
\de (x_0)] } \right)^N |D|^{\frac{1-\al}{N} +\frac{1}{r} +\frac{1}{p} },
\end{equation}
for some positive constant~$k_{N,r,p,\al}$.
When instead $r$, $p$ and~$ \al$ are as in \eqref{eq:conditionBS}, we have that
\begin{equation}\label{eq:estimate_mupp_meanPoincare}
\ol{\mu}_{p,p,0} (D)^{-1} \le k_{N, \, p} \,  \frac{d_D^{3N(1 + \frac{N}{p}) + 1 }  }{
r_i^{3N(1 + \frac{N}{p})}  }
\end{equation}
and
\begin{equation}\label{eq:estimate_mupp_Poincarex_0}
\mu_{p,p, 0}(D, x_0 )^{-1} \le k_{N, \, p} \, \frac{d_D^{3N(1 + \frac{N}{p}) + 1 }  }{
\min[r_i, \de ( x_0 )]^{3N(1 + \frac{N}{p})} },
\end{equation}
for some positive constant~$k_{N, \, p}$.
\end{lem}

Lemma~\ref{lem:John-two-inequalities}
follows from~\cite[item(i) of Lemma 2.1 and items~(i) and~(ii) of Remark~2.4]{MP3}.

{F}rom Lemma~\ref{lem:John-two-inequalities} we can derive estimates for
the derivatives of harmonic functions, as stated in the next result
(a proof of this can be found in~\cite[Corollary 2.3]{MP3}).

\begin{cor}\label{cor:JohnPoincareaigradienti}
Let $D \subset \RR^N$ be a bounded domain satisfying the uniform interior sphere condition with radius $r_i$, and let $v$ be a harmonic function in $D$. Consider three real
numbers~$r$, $p$ and~$ \al$ 
satisfying either~\eqref{eq:conditionHS} or~\eqref{eq:conditionBS}.

(i) If $x_0$ is a critical point of $v$ in $D$, then it holds that
\begin{equation*}
\nr \na v \nr_{r, D } \le \mu_{r, p, \al} ( D , x_0 )^{-1} \nr \de^{\al} \, \na^2 v  \nr_{p, D }.
\end{equation*}

(ii) If
$$\int_{D } \na v(x) \, dx = 0,$$
then it holds that
\begin{equation*}
\nr \na v \nr_{r, D } \le \ol{\mu}_{r, p, \al} (D)^{-1} \nr \de^{\al} \, \na^2 v  \nr_{p,D }.
\end{equation*}
\end{cor}

\begin{rem}\label{rem:poincare valgono anche per John}
{\rm
For later use, we mention that Lemma \ref{lem:John-two-inequalities} and Corollary \ref{cor:JohnPoincareaigradienti} hold true more in general if the assumption of the uniform interior sphere condition is dropped and replaced by the assumption that $D$ is a John domain (see \cite[Section 3.2]{Pog2} or \cite[Lemma 2.1 and Corollary 2.3]{MP3}): in this case explicit estimates of the relevant constants now depending on the John parameter can be found in \cite[Remark 2.4]{MP3}. 
}
\end{rem}

With the aid of Corollary~\ref{cor:JohnPoincareaigradienti} we now prove
the following lemma, which, together with the forthcoming
Theorem~\ref{thm:stabestimateforpsudodistance}, leads to a stability estimate
in terms of the pseudo-distance introduced in~\eqref{eq:pseudodistance}.

\begin{lem}
\label{lem:genericv-trace inequality}
Let $\Om \setminus \ol{\om} \subset \RR^N$ be a bounded domain
of class~$C^1$
satisfying the uniform interior sphere condition with radius $r_i$, that is \eqref{A4},
and let $v \in C^2 (\ol{\Om} \setminus \om)$ be a harmonic function in $\Om \setminus \ol{\om}$.
Let~$u \in C^1 (\ol{\Om} \setminus \om)$ satisfy~\eqref{eq:problem} and assume that~$u \le 0$ on~$\pa \om$.

(i) If $x_0$ is a critical point of $v$ in $\Om$, then it holds that
\begin{eqnarray*}&&
\int_{\Ga} |\na v|^2 dS_x \le \frac{2  N}{r_i} \left(1+\frac{N}{r_i\, \mu_{2,2, \frac{1}{2} }(\Om \setminus \ol{\om} , x_0 )^2 } \right)  \int_{\Om \setminus \ol{\om} } (-u) |\na^2 v|^2 dx 
\\&&\qquad\qquad\qquad\qquad-
\frac{N}{r_i} \int_{\pa \om} \left\lbrace |\na v|^2 u_{\nu} - 2 u < \na^2 v \na v , \nu > \right\rbrace \, dS_x.
\end{eqnarray*}

(ii) If 
$$\int_{\Om \setminus \ol{\om} } \na v \, dx = 0,$$ 
then it holds that
\begin{eqnarray*}&&
\int_{\Ga} |\na v|^2 dS_x \le \frac{2  N}{r_i} \left(1+\frac{N}{r_i\, \ol{\mu}_{2,2, \frac{1}{2} }(\Om \setminus \ol{\om} )^2 } \right)  \int_{\Om \setminus \ol{\om} } (-u) |\na^2 v|^2 dx 
\\&&\qquad\qquad\qquad\qquad-
\frac{N}{r_i} \int_{\pa \om} \left\lbrace |\na v|^2 u_{\nu} - 2 u < \na^2 v \na v , \nu > \right\rbrace \, dS_x.
\end{eqnarray*}
\end{lem}

\begin{proof}
We begin with the following differential identity:
\begin{equation}
\label{diffidimpr}
\dv\,\{v^2 \na u - u \, \na(v^2)\}= v^2 \De u - u \, \De (v^2)=  v^2 - 2 u \, |\na v|^2,
\end{equation}
that holds in $\Om \setminus \ol{\om} $ for any harmonic function~$v$ in $\Om \setminus \ol{\om} $, if $u$ is satisfies \eqref{eq:problem}.

Next, we integrate~\eqref{diffidimpr} on $\Om \setminus \ol{\om}$ and, by the divergence theorem, we get
\begin{equation*}
\int_{\Ga} v^2 u_{\nu} \, dS_x = \int_{\Om \setminus \ol{\om} } v^2\, dx + 
2 \int_{\Om \setminus \ol{\om} } (-u) |\na v|^2 \,dx - \int_{\pa \om}
\left\lbrace v^2 u_{\nu} - 2 u v v_\nu \right\rbrace \, dS_x  .
\end{equation*}
We use this identity replacing the harmonic function~$v$ with
its derivative~$v_{i}$, and then we sum up over $i=1,\dots, N$.
In this way, we obtain
\begin{equation}\label{fieryu458y45hn}\begin{split}
\int_{\Ga} |\na v|^2 u_{\nu}\, dS_x =\;& 
\int_{\Om \setminus \ol{\om} } |\na v|^2 \,dx + 2
\int_{\Om \setminus \ol{\om} } (-u) |\na^2 v|^2\, dx 
\\&\qquad\qquad-
\int_{\pa \om} \left\lbrace |\na v|^2 u_{\nu} - 2 u < \na^2 v \na v , \nu > 
\right\rbrace \, dS_x  .
\end{split}\end{equation}
We observe that,
by an adaptation of Hopf's lemma
(see~\cite[Theorem 3.10]{MP1}),
the term $u_\nu$ in the left-hand side of~\eqref{fieryu458y45hn}
can be bounded from below by $r_i /N$, namely
\begin{equation}\label{597569865}
u_\nu\ge \frac{r_i}{N} \qquad {\mbox{ on }}\Gamma.
\end{equation}
Hence, we obtain from~\eqref{fieryu458y45hn} that
\begin{equation}\label{u707e6egfgfvvk}\begin{split}
\frac{r_i}{N} \, \int_{\Ga} |\na v|^2 \,
dS_x \le\;& \int_{\Om \setminus \ol{\om} } |\na v|^2 \,dx +
2 \int_{\Om \setminus \ol{\om} } (-u) |\na^2 v|^2\, dx  
\\&\qquad\qquad-
\int_{\pa \om} \left\lbrace |\na v|^2 u_{\nu} - 2 u < \na^2 v \na v , \nu > \right\rbrace \, dS_x.
\end{split}\end{equation}
Now we suppose that~$x_0$ is a critical point of~$v$ in~$\Omega$ and we
use item~(i) in Corollary~\ref{cor:JohnPoincareaigradienti}, applied here
with~$D := \Om \setminus \ol{\om}$, $r:=p:=2$ and $\al:=1/2$, and we deduce
from~\eqref{u707e6egfgfvvk} that
\begin{equation*}\begin{split}
\int_{\Ga} |\na v|^2 \,
dS_x \le\;& \frac{N}{r_i\,\mu_{2,2\frac12}(\Omega\setminus\overline\omega,x_0)^{2}}
\int_{\Om \setminus \ol{\om} }\delta |\na^2 v|^2 \,dx +
\frac{2 N}{r_i} \int_{\Om \setminus \ol{\om} } (-u) |\na^2 v|^2\, dx  
\\&\qquad\qquad-
\frac{N}{r_i}\int_{\pa \om} \left\lbrace |\na v|^2 u_{\nu} - 2 u < \na^2 v \na v , \nu > \right\rbrace \, dS_x.
\end{split}\end{equation*}
{F}rom this and~\eqref{eq:relationdist}, one obtains the desired estimate in item~(i). 
In a similar way, using item~(ii) in Corollary~\ref{cor:JohnPoincareaigradienti}, one
shows item~(ii) here, thus completing the proof.
\end{proof}

Now,
we turn our attention to the harmonic function 
\begin{equation}\label{DEh} h := q-u
, \end{equation}
where
\begin{equation}
\label{quadratic}
q(x):=\frac{1}{2N}\, (|x-z|^2-a),
\end{equation}
for some choice of $z\in\RR^N$ and $a\in\RR$.

We remark that, by a direct computation, it is easy to check that
the Cauchy-Schwarz deficit appearing in the left-hand side
of~\eqref{eq:FIdconsovradeterminazione} can be written in terms of $h$ as
\begin{equation}\label{eq:CSintermsofh}
| \na^2 h |^2 = | \na^2 u|^2- \frac{1}{N} = | \na^2 u|^2- \frac{(\De u)^2}{N}.
\end{equation}

Now we specify the choice of the point $z$ in \eqref{quadratic} as follows
\begin{equation}\label{eq:choicez_baricenter}
z:= \frac{1}{|\Om \setminus \ol{\om}|} \left\lbrace \int_{\Om \setminus \ol{\om}} x \, dx - N \int_{\pa \om} u\;\nu \, dS_x \right\rbrace.
\end{equation}
We notice that as 
$\ol{\om}$ tends to the empty set and $\int_{\pa \om} (-u) \, dS_x$ tends to 0, $z$ tends to the baricenter of $\Om$
(however, $z$ is {\it not} the baricenter of~$\Om \setminus \ol{\om}$).

With this choice of $z$ we have that 
\begin{equation}\label{7ujw64rt}
\int_{\Om \setminus \ol{\om} } \na h \, dx = 0 .
\end{equation}
Indeed, 
by a direct computation we get that
\begin{equation}\label{u9iwhschhss}\na h = \frac{(x-z)}{N} - \na u ,\end{equation}
and therefore,
using Green's identity and the fact that $u=0$ on $\Ga$,
\begin{eqnarray*}&&
\int_{\Om \setminus \ol{\om} } \na h \, dx = \int_{\Om \setminus \ol{\om} } \frac{(x-z)}{N} \, dx - \int_{\Om \setminus \ol{\om} } \na u \, dx  
\\&&\qquad=
\int_{\Om \setminus \ol{\om} } \frac{x}{N} \, dx - \frac{z}{N}|\Om \setminus \ol{\om}| - \int_{\pa \om} u\;\nu \, dS_x = 0,
\end{eqnarray*}
thus proving~\eqref{7ujw64rt}.

Gathering the previous results, we thus obtain the desired
estimate on the pseudo-distance:

\begin{thm}\label{thm:stabestimateforpsudodistance}
Let $u \in C^2 (\ol{\Om} \setminus \om)$ satisfy \eqref{eq:problem} and \eqref{eq:overdetermination}, and assume
that~$u \le 0$ on~$\pa \om$.
Let assumptions~\eqref{A3} and~\eqref{A4} be verified.

Then, with the notation of~\eqref{quadratic} and~\eqref{eq:choicez_baricenter},
we have that
\begin{eqnarray*}&&
\int_{\Ga}  \left| \frac{|x-z|}{N} - c \right|^2 \, dS_x \\&\le& 
\frac{N}{r_i} \left(1+\frac{N}{r_i\, \ol{\mu}_{2,2, \frac{1}{2} }(\Om \setminus \ol{\om} )^2 } \right) \Biggl\{ 
 \int_{ \pa \om} \biggl[ c^2  \left( \frac{ < x, \nu>}{N} - u_\nu \right)  +
 2u \left( \frac{< x, \nu >}{N} - u_\nu \right)  
\\&&\quad+
 u_\nu | \na u |^2 - 2 \frac{< x, \na u >}{N} u_\nu +
 | \na u|^2 \frac{< x , \nu >}{N} + \frac{2}{N} u u_\nu - 2 <\na^2 u \na u , \nu > u 
 \biggr] \,dS_x \, \Biggr\} 
\\&&\quad-
\frac{N}{r_i} \int_{\pa \om} \left\lbrace |\na h|^2 u_{\nu} - 2 u < \na^2 h \na h , \nu > 
\right\rbrace \, dS_x
.\end{eqnarray*}
\end{thm}

\begin{proof}
By~\eqref{u9iwhschhss} and the Cauchy-Schwarz inequality, 
$$
\left| \frac{|x-z|}{N} - |\na u| \right| \le | \na h | .
$$
Hence, by using \eqref{eq:overdetermination}, we get that
\begin{equation}\label{jfierut8y6yhr}
\int_{\Ga}  \left| \frac{|x-z|}{N} - c \right|^2 dS_x \le \int_{\Ga} |\na h|^2 \, dS_x.
\end{equation}
Now we remark that we are in the position of using point~(ii)
of Lemma~\ref{lem:genericv-trace inequality} with~$v:=h$, thanks to~\eqref{7ujw64rt}.
Hence, putting together~\eqref{jfierut8y6yhr},
Lemma~\ref{lem:genericv-trace inequality}, \eqref{eq:CSintermsofh}
and \eqref{eq:FIdconsovradeterminazione} we get the desired result.
\end{proof}

\section{Some estimates on \texorpdfstring{$\rho_e - \rho_i$}{rhoe - rhoi}}\label{I:E}
If~$z$ is as in~\eqref{eq:choicez_baricenter}, we set
\begin{equation}\label{eq:rhoe e rhoi}
\rho_e:= \max_{x \in \Ga}{|x-z|} \qquad{\mbox{and}}\qquad \rho_i:=\min_{x \in \Ga}{|x-z|} .
\end{equation}
In this way, whenever $z \in \Om$, we have that
\begin{equation}\label{eq:PA}
B_{\rho_i} (z) \subset \Om \subset B_{\rho_e}(z)  \quad \text{ and } \quad \Ga \subset \ol{B}_{\rho_e}(z) \setminus B_{\rho_i} (z) .
\end{equation}

The aim of the present section is 
%
%
to obtain quantitative estimates for the difference $\rho_e - \rho_i$.

We remark that, recalling the notation in~\eqref{DIAMET},
\begin{equation}\label{eq:Ri}
\rho_e-\rho_i\le d_\Omega.
\end{equation}
Indeed, for every~$w_1$, $w_2\in\Gamma$, we have that
$$ |w_1-z|\le |w_1-w_2|+|w_2-z|\le d_\Omega+|w_2-z|.$$
As a result, taking~$w_1$ maximizing the distance to~$z$,
and~$w_2$ minimizing the distance to~$z$, we obtain that~$\rho_e\le
d_\Omega+\rho_i$, that is~\eqref{eq:Ri}.

Also, by using $\de_\Ga (x)$ to denote the distance of a point $x \in \Om$ to the boundary $\Ga$ we define the {\it complementary parallel set} as
\begin{equation}\label{def:complementary parallel set}
\Om^c_\si:=\{ y\in\Om: \de_\Ga (y) < \si\} \quad \mbox{ for } 0<\si \le r_i.
\end{equation}
Notice that, since $\Om \setminus \ol{\om}$ satisfies the uniform interior sphere
condition of radius $r_i$, it holds that
$$\Om^c_\si \subset \Om \setminus \ol{ \om }\quad \mbox{ for every } 0<\si \le r_i.$$

Lemma \ref{lem:Lp-estimate-oscillation-generic-v} below
contains an inequality  for the oscillation of a harmonic function $v$ in terms of its $L^p$-norm in $\Om \setminus \ol{\om} $ and of a bound for its gradient in $\Om^c_\si$.
More precisely, recalling the notation in~\eqref{media}
and~\eqref{ellep}, we have:

\begin{lem}
\label{lem:Lp-estimate-oscillation-generic-v}
Let $\Om \setminus \ol{\om} \subset\RR^N$
satisfy the uniform interior sphere condition of radius $r_i$ on~$\Ga$,
that is~\eqref{A4bis}, and suppose that~$\Ga$
is of class~$C^1$.
Let $v$ be a harmonic function in $\Om \setminus \ol{\om} $ of
class $C^1 (\ol{\Om^c_{r_i}})$, and let~$G$ be an upper bound for the gradient of $v$ on $\ol{\Om^c_{r_i}}$.
\par
Then, given~$p\ge1$, there exist two positive
constants $a_{N,p}$ and $\al_{N,p}$ depending only on $N$ and $p$ such that if
\begin{equation}
\label{smallness-generic-v}
\nr v - v_{\Om \setminus \ol{\om} } \nr_{p, \Om \setminus \ol{\om} } \le \al_{N,p} \, r_{i}^{\frac{N+p}{p}}  G  ,
\end{equation}
then we have that
\begin{equation}
\label{Lp-stability-generic-v}  
\max_{\Ga} v - \min_{\Ga} v \le a_{N,p} \,  G^{ \frac{N}{N+p} } \, \nr v - v_{\Om \setminus \ol{\om} } \nr_{p, \Om \setminus \ol{\om} }^{ p/(N+p) }.
\end{equation} 
\end{lem}

Lemmata~\ref{lem:Lp-estimate-oscillation-generic-v} and~\ref{lem:L2-estimate-oscillation}
and Theorem~\ref{thm:serrin-W22-stability} here adapt to the present situation ideas originating from~\cite{Pog2} and \cite{MP3} -- see also \cites{MP4} for generalizations in other directions of those ideas. 
Here, we obtain Lemma~\ref{lem:Lp-estimate-oscillation-generic-v} as an immediate consequence of the following new refined estimate, that will be crucial in Subsection \ref{subsec:differentchoicesofz}.

\begin{lem}
\label{lem:refined_Lp-estimate-oscillation-generic-v}
Let $\Om \setminus \ol{\om} \subset\RR^N$
satisfy the uniform interior sphere condition of radius $r_i$ on~$\Ga$,
that is~\eqref{A4bis}, and suppose that~$\Ga$
is of class~$C^1$.
Let $v$ be a harmonic function in $\Om \setminus \ol{\om} $ of
class $C^1 (\ol{\Om^c_{r_i}})$, and let~$G$ be an upper bound for the gradient of $v$ on $\ol{\Om^c_{r_i}}$.

Given $\la \in \RR$, we choose $\ol{x} \in \Ga$ for which 
\begin{equation}\label{eq:xsoprasegnatoilmax}
|v(\ol{x}) - \la| = \max_\Ga  |v(x) - \la| 
\end{equation}
and set
\begin{equation}\label{eq:x0puntoperlemmarefined}
x_0 := \ol{x} - r_i \nu( \ol{x}).
\end{equation}

\par
Then, given~$p\ge1$, there exist two positive
constants $a_{N,p}$ and $\al_{N,p}$ depending only on $N$ and $p$ such that if
\begin{equation}
\label{eq:suppalla_refined-smallness-generic-v}
\nr v - \la \nr_{p, B_{r_i}(x_0) } \le \al_{N,p} \, r_{i}^{\frac{N+p}{p}}  G  .
\end{equation}
then we have that
\begin{equation}
\label{eq:suppalla_refined-Lp-stability-generic-v}  
\max_{\Ga} v - \min_{\Ga} v \le a_{N,p} \,  G^{ \frac{N}{N+p} } \, \nr v - \la \nr_{p,  B_{r_i}(x_0) }^{ p/(N+p) }
\end{equation}
\end{lem}

\begin{proof}
By \eqref{eq:xsoprasegnatoilmax}, it holds that
\begin{equation}\label{eq:miserveperoscnuovolemma}
\max_{\Ga} v - \min_{\Ga} v \le 2 \, | v(\ol{x}) - \la |
\end{equation}

For $0<\si \le r_i$, we
define
$$ \ol{y}:= \ol{x}-\si\nu( \ol{x}) .$$

Notice that, in light of \eqref{eq:x0puntoperlemmarefined} and \eqref{A4bis} -- and being $\Ga$ of class $C^1$ -- we have that
\begin{equation}\label{riprova-jiurherher607687657}
B_\sigma(\ol{y})\subset B_{r_i}(x_0) \subset \Omega\setminus\overline\omega.
\end{equation}

By the fundamental theorem of calculus we have that
\begin{equation}\label{eq:riprova-prova-TFCI-generic v}
v(\ol{x})= v(\ol{y}) + \int_0^\si \lan\na v( \ol{x}-t\nu( \ol{ x} )),\nu( \ol{x} )\ran\,dt .
\end{equation}

Furthermore, since $v$ is harmonic in $\Om \setminus \ol{\om}$,
we can use the mean value property for the balls with radius $\si$ centered at~$\ol{y}$,
thanks to~\eqref{riprova-jiurherher607687657}, and
find that
\begin{eqnarray*}
|v( \ol{y} ) - \la | 
&=&\left| \frac1{|B_1|\, \si^N}\,\int_{B_\si( \ol{y} )} v(y)\,dy- \la \right|
\\
&\le&
\frac1{|B_1|\, \si^N}\,\int_{B_\si( \ol{y} )}|v - \la |\,dy\\&\le&
\frac{1}{ \left[ |B_1|\, \si^N \right]^{1/p} } \, \left[\int_{B_\si( \ol{y} )}|v - \la |^p\,dy\right]^{1/p}\\&\le& 
\frac1{ \left[ |B_1|\, \si^N \right]^{1/p} } \, \left[\int_{B_{r_i}(x_0) }|v - \la |^p\,dy\right]^{1/p} 
,\end{eqnarray*}
where we used an application of H\"older's inequality and~\eqref{riprova-jiurherher607687657} once again.

This, \eqref{eq:miserveperoscnuovolemma} and \eqref{eq:riprova-prova-TFCI-generic v} yield that
\begin{multline}\label{riprova-56}
\max_{\Ga} v - \min_{\Ga} v \le 2 \, |v( \ol{x} ) - \la|
\\
= 2 \, \left| v( \ol{y} ) - \la + \int_0^\si \lan\na v( \ol{ x} -t\nu( \ol{x} )),\nu( \ol{x} )\ran\,dt \right|
\le
2 \left( |v( \ol{y} )- \la | + \sigma G \right)
\\
\le 2 \, \left[  \frac{\nr v - \la \nr_{p, B_{r_i}(x_0) } }{ |B_1|^{1/p} \, \si^{N/p}}+ \si \, G \right] ,
\end{multline}
for every $0<\si \le r_i$.

Therefore, by minimizing the right-hand side of the last inequality, we can conveniently choose 
\begin{equation}
\label{eq:riprova-costsicv}
\si:=\left(\frac{N\,\nr v - \la \nr_{p, B_{r_i}(x_0) }}{ p \, |B_1|^{1/p}\, G }\right)^{p/(N+p)} 
\end{equation}
and obtain \eqref{eq:suppalla_refined-Lp-stability-generic-v}
if $\si \le r_i$; \eqref{eq:suppalla_refined-smallness-generic-v} follows. 

The computations show that
\begin{equation}
\label{eq:costantia_Nal_Nlemmagenericv}
a_{N,p}:= \frac{ 2 (N+p) }{N^{\frac{N}{N+p}} p^{\frac{p}{N+p}} |B_1|^{\frac{1}{N+p}}} 
\quad \mbox{and} \quad \al_{N,p}:= \frac{ p }{N} \, |B_1|^{\frac{1}{p} } .
\end{equation}
\end{proof}

{F}rom Lemma~\ref{lem:refined_Lp-estimate-oscillation-generic-v},
we immediately get the proof of Lemma~\ref{lem:Lp-estimate-oscillation-generic-v}
as follows:

\begin{proof}[Proof of Lemma \ref{lem:Lp-estimate-oscillation-generic-v}]
Since, by \eqref{riprova-jiurherher607687657}
$$
\nr v - \la \nr_{p, B_{r_i}(x_0) }\le \nr v - \la \nr_{p, \Om \setminus \ol{\om} } ,
$$
the desired result easily follows from
Lemma \ref{lem:refined_Lp-estimate-oscillation-generic-v}, by choosing $\la:= v_{\Om \setminus \ol{\om} }$.
The constants~$a_{N,p}$ and $\al_{N,p}$ are still those defined in \eqref{eq:costantia_Nal_Nlemmagenericv}.
\end{proof}

\medskip

We now turn our attention to the harmonic function $h$
introduced in~\eqref{DEh}, and we modify Lemma~\ref{lem:Lp-estimate-oscillation-generic-v} to
link $\rho_e - \rho_i$ to the $L^p$-norm of $h$.
Since $h=q$ on $\Ga$, we have that
\begin{equation}\label{oscillation-P}\begin{split}&
\max_{\Ga} h-\min_{\Ga} h 
=\max_{\Ga} q-\min_{\Ga} q
=\frac{1}{2N} \left(\max_{x\in\Ga} |x-z|^2-\min_{x\in\Ga} |x-z|^2
\right)\\&\qquad\qquad
= \frac1{2N}\,(\rho_e^2-\rho_i^2),
\end{split}\end{equation}
due to~\eqref{quadratic} and~\eqref{eq:rhoe e rhoi}.

We also observe that, by definition of $\rho_e$, it follows that
\begin{equation}\label{Nuovo7scla}
\rho_e \ge \frac{d_\Om }{2}.
\end{equation}
%
%
%

Then, from~\eqref{oscillation-P} and~\eqref{Nuovo7scla} we obtain that
\begin{equation}\label{oscillation}
\max_{\Ga} h-\min_{\Ga} h\ge \frac{d_\Om}{4N} \, (\rho_e-\rho_i).
\end{equation}
The next result gives an explicit bound on the difference $\rho_e -\rho_i$:

\begin{lem}
\label{lem:L2-estimate-oscillation}
Let $\Om \setminus \ol{\om} \subset\RR^N$ satisfy
the uniform interior sphere condition of radius $r_i$ on $\Ga$,
that is~\eqref{A4bis}, and suppose that~$\Ga$
is of class~$C^1$.
Let~$u$ satisfy \eqref{eq:problem} and $u \in C^1 \left( \left( \Om \setminus \ol{\om} \right) \cup \Ga \right)$, let~$q$
be as in \eqref{quadratic} with $z\in\Om$, and let~$h$ be as in~\eqref{DEh}.

Then, there exists a positive constant $C$ such that
\begin{equation}
\label{L2-stability}
\rho_e-\rho_i\le C \, \nr h - h_{\Om\setminus \overline\omega} \nr_{p, \Om\setminus
\overline\omega}^{ p/(N+p) }.
\end{equation}
The constant $C$ depends on $N$, $p$, $d_\Om$, $r_i$, $M$, where 
\begin{equation}
\label{bound-gradient}
M:=\max_{\ol{\Om^c_{r_i}}} |\na u|.
\end{equation}
\end{lem}

\begin{proof}
By direct computations (see e.g.~\eqref{u9iwhschhss})
it is easy to check that
\begin{equation*}
| \na h | \le M + \frac{d_\Om}{N} \quad \mbox{on } \ol{\Om^c_{r_i}},
\end{equation*}
where $M$ is defined in \eqref{bound-gradient}.

We now 
consider
the constants ~$a_{N,p}$ and $\al_{N,p}$ defined in \eqref{eq:costantia_Nal_Nlemmagenericv}
and we
distinguish two cases, according on whether
\begin{equation}
\label{smallness}
\nr h - h_{\Om\setminus\overline\omega} \nr_{p, \Om\setminus
\overline\omega} \le \al_{N,p} \, \left(M + \frac{d_\Om}{N}\right) \, r_{i}^{\frac{N+p}{p}} 
\end{equation}
or
\begin{equation}\label{smallness-no}
\nr h - h_{\Om\setminus\overline\omega} \nr_{p, \Om
\setminus\overline\omega} > \al_{N,p} \, \left(M + \frac{d_\Om}{N}\right) \, r_{i}^{\frac{N+p}{p}} ,
\end{equation}
If~\eqref{smallness} holds true,
we can apply Lemma~\ref{lem:Lp-estimate-oscillation-generic-v} with $v:=h$ and $G:= M + \frac{d_\Om}{N}$.
Thus, by means of \eqref{oscillation} we deduce that
\eqref{L2-stability} holds with 
\begin{equation}\label{eq:constantCmaxnuovolemmaoscillation}
C:= 4 N \, a_{N,p} \, \frac{ \left(M + \frac{d_\Om}{N} \right)^{ \frac{N}{N+p} } }{ d_\Om } .
\end{equation} 
 
On the other hand, if~\eqref{smallness-no}
holds true,
it is trivial to check that \eqref{L2-stability} is verified with
$$C:= \frac{d_\Om}{ \left[ \al_{N,p} \, \left(M + \frac{d_\Om}{N} \right) \right]^{\frac{p}{N+p}} \, r_{i} },$$
thanks to~\eqref{eq:Ri}.

Thus, \eqref{L2-stability} always holds true if we choose
the maximum between this constant and that in \eqref{eq:constantCmaxnuovolemmaoscillation}.
%
%
\end{proof}

\begin{thm}
\label{thm:serrin-W22-stability} 
Let $\Om \setminus \ol{\om} \subset\RR^N$ satisfy the uniform interior sphere
condition of radius $r_i$,
that is~\eqref{A4}, and suppose that~$\Ga$
is of class~$C^1$.
Let~$u$ satisfy \eqref{eq:problem}, $u \in C^1 \left( \left( \Om \setminus \ol{\om} \right) \cup \Ga \right)$, and suppose that~$u \le 0$ on $\pa \om$.
Let~$q$
be as in \eqref{quadratic} with $z$ chosen as in \eqref{eq:choicez_baricenter},
and assume that~$z$ belongs to $\Om$.
Let~$h$ be as in~\eqref{DEh}.

Then, there exists a positive constant $C$ such that
\begin{equation}\label{eq:C-provastab-serrin-W22}
\rho_e-\rho_i\le C\, \nr \de^{\frac{1}{2} } \, \na^2 h  \nr_{2,\Om \setminus \ol{ \om } }^{\tau_N} ,
\end{equation}
with the following specifications:
\begin{enumerate}[(i)]
\item $\tau_2 = 1$;
\item $\tau_3$ is arbitrarily close to~$1$, in the sense that, for any $\theta\in(0,1)$ sufficiently
small, there exists a positive constant $C$ such that  \eqref{eq:C-provastab-serrin-W22} holds with $\tau_3 = 1- \theta$;
\item $\tau_N = 2/(N-1)$ for $N \ge 4$.
\end{enumerate}

The constant $C$
depends on $N$, $r_i$, $d_\Om$, $M$ (as
defined in \eqref{bound-gradient}),
and $\theta$ (the latter, only in the case $N=3$).
\end{thm}

\begin{proof}
For the sake of clarity, we will always use the letter $C$ to denote the constants in all the inequalities appearing in the proof.
Their explicit computation will be clear by following the steps of the proof
(see the forthcoming Remark~\ref{rem:dipendenzecostanti}).
\medskip

(i) Let $N=2$.
By the Sobolev immersion theorem (for instance we apply~\cite[Theorem~9.1]{Fr} to the
function~$h-h_\Om$), we deduce that there exists
a positive constant~$C$ such that
\begin{equation}
\label{eq:immersionSerrinN2VERSIONENEW}
\max_{\ol{\Om}  \setminus \om  } | h - h_{\Om  \setminus \ol{ \om } } |  \le C \, \nr h-h_{\Om  \setminus \ol{ \om } } \nr_{W^{1,4}(\Om  \setminus \ol{ \om } )} .
\end{equation}
As noticed in~\cite[Remark 2.9]{MP3},
the immersion constant in \eqref{eq:immersionSerrinN2VERSIONENEW} depends on $N$ and $r_i$ only.

Applying \eqref{John-harmonic-poincare} with $D:= \Om \setminus \ol{ \om }$,
$v:=h$, $r:=p:=4$, and $\al:=0$ leads to
\begin{equation}\label{fiertu845mdkfn}
\nr h - h_{ \Om \setminus \ol{ \om } } \nr_{W^{1,4}( \Om \setminus \ol{ \om } )}\le C \, \nr \na h\nr_{4, \Om \setminus \ol{ \om } }.
\end{equation}
Also, since \eqref{7ujw64rt} holds true, we can apply item (ii) of
Corollary~\ref{cor:JohnPoincareaigradienti} with~$v:=h$,
$D:= \Om \setminus \ol{ \om }$, $r:=4$, $p:=2$, and $\al:=1/2$
and obtain that
$$
\nr \na h \nr_{4, \Om \setminus \ol{ \om } } \le C \,  \nr \de^{\frac{1}{2} } \, \na^2 h  \nr_{2, \Om \setminus \ol{ \om } } .
$$
{F}rom this and~\eqref{fiertu845mdkfn}, we get that
$$
\nr h - h_{ \Om \setminus \ol{ \om } } \nr_{W^{1,4}( \Om \setminus \ol{ \om } )}\le C \, \nr \de^{\frac{1}{2} } \, \na^2 h \nr_{2, \Om \setminus \ol{ \om } } .
$$
This inequality, together with \eqref{eq:immersionSerrinN2VERSIONENEW}, gives that
\begin{equation*}
\max_\Ga h-\min_\Ga h \le 
C \, \nr \de^{\frac{1}{2} } \, \na^2 h  \nr_{2, \Om \setminus \om }.
\end{equation*}
Thus, by recalling \eqref{oscillation} we get that \eqref{eq:C-provastab-serrin-W22} holds with $\tau_2=1$.

\medskip

(ii) Let $N=3$.
For any~$\theta\in(0,1)$ sufficiently small, we notice that
$$ r:= \frac{3(1- \theta)}{\theta},\qquad
p:=3(1 -\theta )\quad {\mbox{ and }} \quad \al:=0$$
satisfy~\eqref{eq:conditionHS} in Lemma~\ref{lem:John-two-inequalities}.
Hence, we can apply the estimate in~\eqref{John-harmonic-poincare} with
$D:= \Om \setminus \ol{ \om }$ and~$v:=h$, obtaining that
\begin{equation}\label{ioeru4bndftrjt}
\nr h - h_{\Om \setminus \ol{\om} } \nr_{\frac{3( 1 - \theta)}{\theta}, 
\Om \setminus \ol{\om} } \le C \,\nr \nabla h\nr_{2,\Om \setminus \ol{\om} }.\end{equation}
Furthermore,
$$ r:=3 ( 1 - \theta ),\qquad p:=2,\quad {\mbox{ and }}\quad \al:=\frac12$$
satisfy~\eqref{eq:conditionHS}, and therefore
item (ii) of
Corollary~\ref{cor:JohnPoincareaigradienti}, applied again with~$D:= \Om 
\setminus \ol{ \om }$ and~$v:=h$, yields that
$$
\nr \nabla h\nr_{2,\Om \setminus \ol{\om} }\le
C \, \nr \de^{\frac{1}{2} } \, \na^2 h \nr_{2,\Om \setminus \ol{\om} }.
$$
This and~\eqref{ioeru4bndftrjt} give that
$$ \nr h - h_{\Om \setminus \ol{\om} } \nr_{\frac{3( 1 - \theta)}{\theta}, 
\Om \setminus \ol{\om} } \le 
C \, \nr \de^{\frac{1}{2} } \, \na^2 h \nr_{2,\Om \setminus \ol{\om} }.$$
Thus, by using Lemma~\ref{lem:L2-estimate-oscillation} with~$p:=\frac{3( 1 - \theta)}{\theta}$,
we obtain that
$$ \rho_e-\rho_i \le C \, \nr \de^{\frac{1}{2} } \, \na^2 h \nr_{2,\Om \setminus \ol{\om} }^{1-\theta}.$$
which is~\eqref{eq:C-provastab-serrin-W22} with $\tau_3 = 1- \theta $. 
\medskip

(iii) Let $N \ge 4$. In light of~\eqref{7ujw64rt},
we can apply to $h$ item (ii) of Corollary~\ref{cor:JohnPoincareaigradienti}
with $D:= \Om \setminus \ol{ \om }$, $r:=\frac{2N}{N-1}$, $p:=2$, and~$\al:=1/2$
(noticing that they satisfy~\eqref{eq:conditionHS}), and obtain that
\begin{equation}\label{doetutyty}
\nr \na h \nr_{\frac{2N}{N-1}, \Om \setminus \ol{\om} } \le C \, \nr \de^{\frac{1}{2} } \, \na^2 h  \nr_{2,\Om \setminus \ol{\om} } .
\end{equation}
Being $N \ge 4$, we can also
apply \eqref{John-harmonic-poincare} with $D:= \Om \setminus \ol{ \om }$, $v:=h$, 
$r:=\frac{2N}{N-3}$, $p:=\frac{2N}{N-1}$, and~$\al:=0$, and get
\begin{equation}\label{doetutyty2}
\nr h- h_{ \Om \setminus \ol{\om} } \nr_{\frac{2N}{N-3} } \le C \, \nr \na h \nr_{\frac{2N}{N-1},\Om \setminus \ol{\om} }.
\end{equation}
Thus, from~\eqref{doetutyty} and~\eqref{doetutyty2} we conclude that
$$
\nr h- h_{ \Om \setminus \ol{\om} } \nr_{\frac{2N}{N-3} , \Om \setminus \ol{\om}  } \le C \, \nr \de^{\frac{1}{2} } \, \na^2 h  \nr_{2,\Om \setminus \ol{\om} }.
$$
Then, Lemma~\ref{lem:L2-estimate-oscillation}, applied with~$p:=\frac{2N}{N-3}$, gives that~\eqref{eq:C-provastab-serrin-W22}
holds with $\tau_N = 2/(N-1)$.
\end{proof}

\begin{rem}[On the constant $C$ in~\eqref{eq:C-provastab-serrin-W22}]
\label{rem:dipendenzecostanti}
{\rm
The constant $C$ in~\eqref{eq:C-provastab-serrin-W22}
can be explicitly computed by following the steps of the proof
of Theorem~\ref{thm:serrin-W22-stability}, and it can be shown to depend only on the parameters mentioned in the statement of Theorem~\ref{thm:serrin-W22-stability}. 
Indeed, the parameters $\ol{\mu}_{r,p,\al} (\Om \setminus \ol{\om} )$ and $\ol{\mu}_{p,p,\al} (\Om \setminus \ol{\om} )$,  can be estimated by means of \eqref{eq:estimatePoincaremean} and \eqref{eq:estimate_mupp_meanPoincare}.
Then, we notice that $d_{\Om \setminus \ol{\om}} \le d_\Om$
and 
\begin{equation}\label{eq:nuovadarichiamareinsiemeainequality_volume_diameter}
| \Om \setminus \ol{\om} | \le | \Om |. 
\end{equation}
Finally, to remove the dependence on the volume,
we use the trivial bound
\begin{equation}\label{eq:inequality_volume_diameter}
|\Om|^{1/N} \le |B_1|^{1/N} d_\Om /2 .
\end{equation}
}
\end{rem}

\begin{rem}[Another choice for the point $z$ in~\eqref{eq:rhoe e rhoi}] \label{rem:choices z}
{\rm
Another possible way to choose $z$ in~\eqref{eq:rhoe e rhoi}
(different from~\eqref{eq:choicez_baricenter})
is 
$$z= x_0 - \na u (x_0) ,$$
where $x_0 \in \Om \setminus \ol{\om}$ is any point such that $\de (x_0) \ge r_i$. In fact, with this choice we obtain that $\na h (x_0)=0$ and we can thus use item (i) of Corollary \ref{cor:JohnPoincareaigradienti} instead of item (ii).

Then, we can estimate the parameters $\mu_{r,p,\al}
(\Om \setminus \ol{\om}, x_0)$ in terms of $r_i$ and
$d_\Om$ by using \eqref{eq:estimatePoincarex_0},
the fact that $\de (x_0) \ge r_i$, and \eqref{eq:inequality_volume_diameter}.  
\par
Also with this choice, in order to obtain an analogue
of Theorem~\ref{thm:serrin-W22-stability}, we should additionally
require that~$z \in \Om$, to be sure that the ball $B_{\rho_i}(z)$ is
contained in $\Om$.
}
\end{rem}

\section{Stability results and proof of Theorem~\ref{MAIN}}\label{SEC-6}

For the sake of clarity, we state here some notation that
will be used throughout the rest of this paper.

We denote by~$\bar d_\om$ the supremum of the diameters
of all the connected components of~$\omega$ (of course,
if~$\omega$ is connected, then~$\bar d_\om$ coincides with
the diameter of~$\omega$, and, in this case, according to the notation in~\eqref{DIAMET}, it holds that~$\bar d_\om=d_\om$). Then, we have:

\begin{thm}[General stability result for $\rho_e - \rho_i$]
\label{thm:stability_radii}
Let $u \in C^2 (\ol{\Om} \setminus \om)$ satisfy \eqref{eq:problem} and \eqref{eq:overdetermination}, and
suppose that~$u \le 0$ on $\pa \om$.
Let assumptions~\eqref{A3} and~\eqref{A4} be verified.
Assume also that the point $z$ chosen in~\eqref{eq:choicez_baricenter} belongs to~$\Om$.

If $\psi:[0,\infty)\to[0,\infty)$ is a continuous function vanishing at $0$ such that
\begin{equation}\label{eq:assumptions_stabgeneral}
\left. \begin{aligned}
\int_{\pa \om} (-u) \, dS_x = \int_{\pa \om} |u| \, dS_x
\\
\left| \int_{\pa \om}  u \, u_\nu \, dS_x \right|
\\
\left| \int_{\pa \om}  | \na u|^2 \, u_\nu \, dS_x \right|
\\
\int_{\pa \om} | \na u|^2 \, dS_x
\\
\left| \int_{\pa \om} < \na^2 u \, \na u , \nu> \, u \, dS_x \right|
\end{aligned}
\right\rbrace
\le \psi(\eta) \quad \text{ with } \eta= |\pa \om| \, \text{ or } \, \eta = \bar d_\om ,
\end{equation}
then
\begin{equation}
\label{eq:generalstability}
\rho_e-\rho_i\le C \, \psi(\eta)^{\tau_N/2},
\end{equation}
where $\tau_N$ is as in Theorem~\ref{thm:serrin-W22-stability} and $C$ is a positive
constant depending on $N$, $r_i$, $d_\Om$, and~$M$ (as defined in~\eqref{bound-gradient}).
\end{thm}

\begin{proof} Up to a translation,
we can suppose that the origin lies in~$\Omega$,
and therefore
$$|< x, \nu >|\le
|x| \le d_\Om.$$ Hence,
the desired result easily follows by putting together 
Theorem~\ref{thm:serrin-W22-stability} and formulas~\eqref{eq:relationdist}, \eqref{eq:CSintermsofh} and~\eqref{eq:FIdconsovradeterminazione}.
\end{proof}

\begin{rem}\label{REMA72}
{\rm We notice that
with the choice of~$z$ as in~\eqref{eq:choicez_baricenter},
(if $\eta$ is small enough)
the assumption $z \in \Om$ is satisfied, at least if the baricenter of~$\Omega$
lies in~$\Omega$ (in particular, if~$\Omega$ is convex).
}
\end{rem}

With this preliminary work, we are now in the position of obtaining a quantitative rigidity
result bounding the averaged squared pseudodistance of the form
$$\int_{\Ga}  \left| \frac{|x-z|}{N} - c \right|^2 dS_x,$$
where~$z$ is as in~\eqref{eq:choicez_baricenter} and $c$ is that in \eqref{eq:overdetermination}. The precise result goes as follows:

\begin{thm}[General stability result for a pseudodistance]
\label{thm:stab_pseudodistance}
Let $u \in C^2 (\ol{\Om} \setminus \om)$ satisfy \eqref{eq:problem} and \eqref{eq:overdetermination},
and suppose that~$u \le 0$ on $\pa \om$.
Let assumptions~\eqref{A3} and~\eqref{A4} be verified, and $z$ be as in~\eqref{eq:choicez_baricenter}.

If $\psi:[0,\infty)\to[0,\infty)$ is a continuous function vanishing at $0$ such
that~\eqref{eq:assumptions_stabgeneral} holds true together with
\begin{equation}\label{eq:assumption_further_stabpseudodistance}
\left| \int_{\pa\om} u \, < \na^2 u \, (x-z) , \nu > \, dS_x \right| \le \psi(\eta) \quad \text{ with } \eta= |\pa \om| \, \text{ or } \, \eta = \bar d_\om ,
\end{equation}
then
\begin{equation}\label{eq:stellanuovamiserve}
\int_{\Ga}  \left| \frac{|x-z|}{N} - c \right|^2 dS_x \le C \, \psi(\eta),
\end{equation}
where $C$ is a constant depending on $N$, $r_i$, $d_\Om$.
\end{thm}

\begin{proof}
In light of~\eqref{DEh}, \eqref{quadratic}, and~\eqref{u9iwhschhss}, we see that
$$
|\na h|^2 u_\nu= \left\lbrace \left(\frac{| x-z |}{N} \right)^2  - \frac{2}{N} <  (x-z) , \na u> + |\na u|^2 \right\rbrace u_\nu 
$$
and
\begin{eqnarray*} &&
< \na^2 h \na h , \nu >\\& =&\,
< \left( \frac{1}{N} I - \na^2 u \right) \left( \frac{ (x-z) }{N} - \na u \right) , \nu> 
\\&=&\,
\frac{1}{N} < \frac{ (x-z) }{N} - \na u , \nu> - < \na^2 u \left( \frac{ (x-z) }{N}
- \na u \right) , \nu >\\
&=&\,
\frac{1}{N^2} < (x-z) , \nu> - \frac{1}{N} < \na u , \nu> -
\frac{1}{N} < \na^2 u \, (x-z) , \nu > + < \na^2 u \na u , \nu >.
\end{eqnarray*}
As a consequence,
$$
\left| \int_{\pa\om} |\na h|^2 u_\nu \, dS_x \right| 
\le \frac{d_\Om^2}{N^2} \int_{\pa\om} | \na u | \, dS_x + \frac{2 d_\Om}{N} 
\int_{\pa\om} |\na u|^2  \, dS_x + \left| \int_{\pa\om} | \na u|^2 u_\nu \, dS_x \right| 
$$
and
\begin{eqnarray*}
&&\left| \int_{\pa\om} u < \na^2 h \na h , \nu > \, dS_x \right|\\
&\le& 
\frac{d_\Om}{N^2}  \int_{\pa\om} |u | \, dS_x  + \frac{1}{N} \left| 
\int_{\pa\om} u \, u_\nu \, dS_x \right| \\&&\qquad+ 
\frac{1}{N} \left| \int_{\pa\om} u \,
< \na^2 u \, (x-z) , \nu > \, dS_x \right| + \left| 
\int_{\pa\om} u \,  < \na^2 u \na u , \nu > \, dS_x \right| .
\end{eqnarray*}
Hence, 
the proof follows from the last two formulas, Theorem~\ref{thm:stabestimateforpsudodistance}, \eqref{eq:assumptions_stabgeneral}, and \eqref{eq:assumption_further_stabpseudodistance}.
\end{proof}

By means of Lemma \ref{lem:relationasymmetrypseudodistance}, from Theorem \ref{thm:stab_pseudodistance} we also obtain
a quantitative rigidity result for the
asymmetry~\eqref{eq:asymmetryallaFraenkel}:

\begin{thm}[General stability result for an asymmetry]\label{thm:stab_Asymmetry}
Let $u \in C^2 (\ol{\Om} \setminus \om)$ satisfy \eqref{eq:problem} and \eqref{eq:overdetermination},
and suppose that~$u \le 0$ on $\pa \om$.
Let assumptions~\eqref{A3} and~\eqref{A4} be verified, and $z$ be as in~\eqref{eq:choicez_baricenter}.

If $\psi:[0,\infty)\to[0,\infty)$ is a continuous function vanishing at $0$ such
that~\eqref{eq:assumptions_stabgeneral} and~\eqref{eq:assumption_further_stabpseudodistance} hold true,
then the asymmetry 
defined in~\eqref{eq:asymmetryallaFraenkel} satisfies
\begin{equation}\label{eq:dithmstab_asymmetry}
\frac{| \Om \De B_{Nc}(z) |}{| B_{Nc}(z) |} \le C \, \psi(\eta)^{1/2},
\end{equation}
where $C$ is a constant depending on $N$, $r_i$, $d_\Om$, and $c$.
\end{thm}

\begin{rem}\label{rem:togliere c quandoabbiamosfera}
{\rm
The dependence on $c$ of the
constant in~\eqref{eq:dithmstab_asymmetry}
can be dropped by exploiting suitable bounds for it. Indeed,
on the one hand, putting together \eqref{eq:overdetermination} and \eqref{597569865} one obtains the lower bound
\begin{equation}\label{eq:00000lowerboundcconsferainterna}
c \ge \frac{r_i}{N} .
\end{equation}
On the other hand, from the expression in \eqref{eq:value_c} one can obtain an upper bound for $c$ in terms of $N$ and $d_\Om$, when $\eta$ is small enough. More precisely,
formula~\eqref{eq:value_c} implies that
\begin{equation}\label{eq:disuguaglianze per upper bound c:0} 
c|\Gamma|= |\Om|-|\omega|-\int_{\partial\omega} u_\nu\,dS_x\le
\frac{ |\Om|-|\omega|}2, \quad \text{if $\eta$ is small enough,}
\end{equation}
thanks to~\eqref{eq:assumptions_stabgeneral}.

Moreover, exploiting the classical isoperimetric inequality
and~\eqref{eq:inequality_volume_diameter}, one sees that
\begin{equation}\label{eq:disuguaglianze per upper bound c}
\frac{|\Om| - | \om | }{| \Ga|} \le \frac{|\Om|}{|\Ga|} \le \frac{1}{N} \left( \frac{|\Om| }{| B_1|} \right)^{\frac{1}{N}} \le \frac{d_\Om}{2 N} .
\end{equation}
Putting together~\eqref{eq:disuguaglianze per upper bound c:0}
and~\eqref{eq:disuguaglianze per upper bound c} one obtains
that
\begin{equation}\label{eq:STARSTARSTARSTAR}
c\le \frac{d_\Om}{4 N} , \quad \text{if $\eta$ is small enough}.
\end{equation}
}
\end{rem}

We can now obtain a quantitative symmetry result by assuming a $C^2$-bound of the solution along~$\partial\omega$:

\begin{thm}\label{JAHS334}
Let $u \in C^2 (\ol{\Om} \setminus \om)$ satisfy \eqref{eq:problem} and \eqref{eq:overdetermination},
and suppose that~$u \le 0$ on $\pa \om$.
Let assumptions \eqref{A3} and \eqref{A4} be verified, and~$z$ be as in~\eqref{eq:choicez_baricenter}.

If there exists $K >0$ such that 
\begin{equation}\label{eq:assumptionC2norm}
\nr u \nr_{C^2 (\pa \om)} \le K ,
\end{equation}
then
\begin{equation}\label{eq:pseudodistancestabconnormaC^2}
\int_{\Ga}  \left| \frac{|x-z|}{N} - c \right|^2 dS_x \le C | \pa \om |,
\end{equation}
where $C$ is a positive constant depending on $N$, $r_i$, $d_\Om$, and $K$.

Also, 
%
%
it holds that
\begin{equation}\label{eq:asymmetrystabconnormaC^2}
\frac{| \Om \De B_{Nc}(z) |}{| B_{Nc}(z) |} \le C | \pa \om |^{1/2},
\end{equation}
where $C$ is a positive constant depending on $N$, $r_i$, $d_\Om$, and $K$.

Moreover, if $z \in \Om$, we have that
\begin{equation}
\label{eq:stabconnormaC^2}
\rho_e-\rho_i\le C |\pa \om|^{\tau_N/2},
\end{equation}
where $\tau_N$ are as in Theorem~\ref{thm:serrin-W22-stability} and $C$ is a positive
constant depending on $N$, $r_i$, $d_\Om$, and~$K$.
\end{thm}

\begin{proof}
We notice that, by~\eqref{eq:assumptionC2norm}, we have that the assumptions
in~\eqref{eq:assumptions_stabgeneral} and \eqref{eq:assumption_further_stabpseudodistance} are satisfied with
$$\psi(|\pa \om|) :=
\max \left\lbrace K, K^3 \right\rbrace |\pa \om|,$$ and so we are in the position
of applying Theorems~\ref{thm:stability_radii}, ~\ref{thm:stab_pseudodistance}, and ~\ref{thm:stab_Asymmetry}, thus obtaining the desired estimates.
\medskip

Notice that the constant in \eqref{eq:stabconnormaC^2} does not depend on $M$
(as defined in~\eqref{bound-gradient}), differently from that appearing in \eqref{eq:generalstability}.
Indeed, we claim that,
when $|\pa \om|<1$, 
\begin{equation}\label{iut87686hrehg}
\max_{ \ol{\Om} \setminus \om } | \na u|\le C,
\end{equation}
for some positive constant~$C$
depending on $N$, $d_\Om$, $r_i$, and $K$.
We point out that, since~$M  \le \max_{ \ol{\Om} \setminus \om } | \na u|$, the estimate in~\eqref{iut87686hrehg}
provides a bound for~$M$ in terms
of~$N$, $d_\Om$, $r_i$, and $K$.
Hence, we now focus on the proof of~\eqref{iut87686hrehg}.

For this, we observe that,
since $|\na u|$ attains its maximum on $\Ga \cup \pa \om$,
recalling~\eqref{eq:overdetermination}
and~\eqref{eq:assumptionC2norm},
we have that
\begin{equation}\label{eq:eqeqeqeqeqeq}
\max_{ \ol{\Om} \setminus \om } | \na u| \le \max \left\lbrace c, \, K \right\rbrace .
\end{equation}
As a result, to obtain the desired estimate
in~\eqref{iut87686hrehg}, it remains
to find an upper bound for $c$ depending on $N$, $d_\Om$, $r_i$, and $K$. 

For this, we notice that,
by \eqref{eq:value_c} and \eqref{eq:assumptionC2norm},
\begin{equation*}
c= \frac{|\Om \setminus \ol{ \om }|}{|\Ga|} - \frac{1}{|\Ga|} \int_{\pa \om} u_\nu \, dS_x 
\le \frac{|\Om |}{|\Ga|} + \frac{K}{|\Ga|} |\pa \om| ,
\end{equation*}
and hence 
\begin{equation}\label{eq:00000000011111111}
c \le  \frac{|\Om |}{|\Ga|} + \frac{K}{|\Ga|} \quad \text{if} \quad |\pa \om| <1  .
\end{equation}
On the one hand, by \eqref{eq:disuguaglianze per upper bound c} we already know that
\begin{equation}\label{eq:00000000011111222}
\frac{|\Om |}{|\Ga|} \le \frac{d_\Om}{2 N} .
\end{equation}
On the other hand, by combining the inequality
$$
|\Om| \ge | B_1 | r_i^{N},
$$
that holds true since a ball of radius $r_i$ is surely contained in $\Om$, with the classical isoperimetric inequality $| \Ga| \ge N |B_1|^{1/N} |\Om|^{(N-1)/N}$,
we get that
\begin{equation}\label{eq:00000001111122233}
\frac{K}{| \Ga |} \le \frac{1}{N |B_1|} \frac{K}{r_i^{N-1}}.
\end{equation}
Putting together \eqref{eq:00000000011111111}, \eqref{eq:00000000011111222},
and \eqref{eq:00000001111122233}, we find the desired explicit upper bound for $c$:
\begin{equation}\label{eq:000upperboundperc}
c \le \frac{d_\Om}{2 N} + \frac{1}{N |B_1|} \frac{K}{r_i^{N-1}}  \quad \text{if} \quad |\pa \om| <1  .
\end{equation}
In turn, this and \eqref{eq:eqeqeqeqeqeq} give that
$$\max_{ \ol{\Om} \setminus \om } | \na u| \le \max \left\lbrace \frac{d_\Om}{2 N} + \frac{1}{N |B_1|} \frac{K}{r_i^{N-1}} , \, K \right\rbrace \quad \text{if} \quad |\pa \om| <1 ,$$
that is the desired estimate in~\eqref{iut87686hrehg}.

The estimate in~\eqref{iut87686hrehg} proves that, if $|\pa \om| <1$, \eqref{eq:stabconnormaC^2} holds true (with $C$ not depending on $M$). On the other hand, if $|\pa \om| \ge 1$, \eqref{eq:stabconnormaC^2} trivially holds true with $C:=d_\Om$, being $\rho_e - \rho_i \le \rho_e \le d_\Om$.
\medskip

Notice also that \eqref{eq:asymmetrystabconnormaC^2} is a global estimate in which the constant does not depend  on $c$, differently from that appearing in \eqref{eq:dithmstab_asymmetry}. 

Indeed, if $|\pa \om| <1$, we can remove the dependence of
the constant on~$c$ thanks to the bounds
in~\eqref{eq:00000lowerboundcconsferainterna}
and~\eqref{eq:000upperboundperc}.

This proves that, if $| \pa \om | <1$, \eqref{eq:asymmetrystabconnormaC^2} holds true with $C$ not depending on $c$. On the other hand, when $|\pa \om| \ge 1$, \eqref{eq:asymmetrystabconnormaC^2} trivially holds true with $C:= \left( \frac{ d_\Om }{ 2 r_i} \right)^N +1 $, being
$$
\frac{|\Om \De B_{Nc}|}{|B_{Nc}|} \le \frac{|\Om|}{|B_{Nc}|} + 1 \le \left( \frac{d_\Om}{2 r_i} \right)^N + 1 ,
$$
where we have used \eqref{eq:00000lowerboundcconsferainterna} and \eqref{eq:inequality_volume_diameter}
in the last inequality.

These observations complete the proof of
Theorem~\ref{JAHS334}.
\end{proof}

We observe that our main result in Theorem~\ref{MAIN}
is now a simple consequence of
Theorem~\ref{JAHS334}.
\medskip

Another instance in which the assumptions of Theorems~\ref{thm:stability_radii}, ~\ref{thm:stab_pseudodistance}, and \ref{thm:stab_Asymmetry}, are surely satisfied is the following.
Notice that in high dimensions (i.e., when~$N>4$) the following result
allows~$u$ to blow up on~$\pa \om$ as either the
perimeter or the diameter tends to zero.

\begin{thm}\label{thm:blowup}
Let $u \in C^2 (\ol{\Om} \setminus \om)$ satisfy \eqref{eq:problem} and \eqref{eq:overdetermination},
and suppose that~$u \le 0$ on $\pa \om$.
Let assumption \eqref{A4} be verified, and~$z$ be as in~\eqref{eq:choicez_baricenter}.  
Let $\om$ be the union of finitely many
disjoint balls
of radius~$\ve$ and assume that 
$$
\|u\|_{L^\infty(\pa\om)}+\ve\|\nabla u\|_{L^\infty(\pa\om)}+
\ve^2\|\nabla^2 u\|_{L^\infty(\pa\om)} = o(\ve^{ \frac{4-N}{3} })
.$$
Then, 
\begin{equation*}
\int_{\Ga}  \left| \frac{|x-z|}{N} - c \right|^2 dS_x  \quad \text{ and } \quad \, \frac{| \Om \De B_{Nc}(z) |}{| B_{Nc}(z) |} 
\end{equation*}
are as small as we wish for small~$\ve$.

Furthermore, if in addition~$z\in \Om$, then also~$
\rho_e-\rho_i$ is as small as we wish for small~$\ve$.
\end{thm}

\section{The case of general domains and proofs of Theorems~\ref{MAIN-DUE} and~\ref{MAIN-TRE}}\label{sec:relaxing assumptions}

As mentioned in the Introduction, analogous stability
results can be obtained by weakening the assumptions in~\eqref{A3} and~\eqref{A4}.

\medskip

Subsection \ref{subsec:John} is devoted to 
the case in which $\Om \setminus \ol{\om}$ is a John domain. 
We extend the stability estimates for the spherical pseudodistance defined in \eqref{eq:pseudodistance} and for the asymmetry defined in \eqref{eq:asymmetryallaFraenkel} -- i.e., Theorems \ref{thm:stab_pseudodistance}, \ref{thm:stab_Asymmetry}, and their corresponding consequences in Theorem \ref{JAHS334} --,
when \eqref{A3} and \eqref{A4}
are dropped and replaced by the weaker assumptions \eqref{A5JOHN}, \eqref{A3bis}; that is,
\begin{equation}\label{eq:questamessadopoassunzionesoloperpseudodistanza}
{\mbox{when $\Om \setminus \ol{\om}$ is a bounded {\it John domain} of {\it finite perimeter}.}}
\end{equation}
%
%
We also show that the pointwise results of Theorem \ref{thm:stability_radii} and its corresponding consequences
in Theorem~\ref{JAHS334}
can be obtained when \eqref{A3} and \eqref{A4}
are replaced with the weaker assumptions \eqref{A5JOHN} \eqref{A3bis}, \eqref{A4bis},
that is,
\begin{equation}\label{otto.1}
\begin{split}
&{\mbox{when $\Om \setminus \ol{\om}$ is a bounded {\em John domain} of {\em
finite perimeter}}}\\&{\mbox{which satisfies the uniform {\em interior sphere condition} on the {\em
external boundary},}}
\end{split}
\end{equation}
at the cost of getting a worse stability exponent $\tau_N$.

All the generalizations presented in Subsection \ref{subsec:John} are obtained by using the same
choice \eqref{eq:choicez_baricenter} for the point $z$. As a particular case of those generalizations, we obtain Theorem \ref{MAIN-DUE}.

\medskip

In Subsection \ref{subsec:differentchoicesofz} we show how a
different choice of the point $z$ allows
to obtain Theorem~\ref{thm:stability_radii} and its corresponding
consequences in Theorem~\ref{JAHS334} in their
full power -- i.e., with~$\tau_N$ given
in Theorem~\ref{thm:serrin-W22-stability} -- under the weaker
assumptions \eqref{A3bis} and \eqref{A4bis}.
We stress that this approach does not need
the assumption \eqref{A5JOHN} that $\Om \setminus \ol{\om}$ is a John domain, requested in the generalizations of Subsection \ref{subsec:John}.
In fact, the set of assumptions on $\Om \setminus \ol{\om}$ in Subsection \ref{subsec:differentchoicesofz} is
\begin{equation}\label{otto.2}
\begin{split}
&{\mbox{$\Om \setminus \ol{\om}$ is a bounded domain of {\em finite perimeter} which
satisfies}}\\&{\mbox{the uniform {\em interior sphere condition} on the {\em external boundary}.}}
\end{split}
\end{equation}

\medskip

We recall that whenever \eqref{eq:overdetermination} is in force, thanks to \eqref{noWea}, the external boundary $\Ga$ is of class $C^{2,\al}$ and $u \in C^{2,\al} \left( \left( \Om \setminus \ol{\om} \right) \cup \Ga \right)$.

\medskip

To deal with sets of finite perimeter, which is common both in~\eqref{otto.1}, \eqref{otto.2}, and \eqref{eq:questamessadopoassunzionesoloperpseudodistanza}, we recall that, thanks to De Giorgi's structure theorem
(see \cite[Theorem 15.9]{Maggi} or \cite{Giusti}), the
assumptions in~\eqref{assumption:regularityuptobordino} and~\eqref{A3bis}
(in the sense explained in Section \ref{NOATZ}) guarantee that the
integral identities proved in Section~\ref{S:1} still hold true, provided that
one replaces~$\pa \om$ with the reduced boundary~$\pa^* \om$ and
agrees to still use $\nu$ to denote the (measure-theoretic)
outer unit normal (see e.g., \cite[Chapter 15]{Maggi}).

We observe that when in particular $\Om \setminus \ol{ \om }$ is of class $C^1$,
that is when~\eqref{A3} is in force,
then~$\pa^* \om = \pa \om$ and the (measure-theoretic)
outer unit normal coincides with the classical notion of outer unit normal
(see~\cite[Remark~15.1]{Maggi}).

Moreover, in the setting of assumption~\eqref{A3bis},
the surface measure $|\pa \om|$
has to be replaced with the {\it perimeter} of $\om$.
We recall indeed that 
the perimeter of $\om$ equals the $N-1$-dimensional Hausdorff
measure of $\pa^* \om$, denoted by $\cH^{N-1}(\pa^* \om)$
(see \cite[Chapter 4]{Giusti} or \cite[Chapter 15]{Maggi}). 
Of course, when $\pa \om$ is of class $C^1$, as given by~\eqref{A3},
those notions agree (since in that case we have $\pa^* \om = \pa \om$).

%
%

\subsection{Generalizations for John domains and proof of Theorem \ref{MAIN-DUE}}\label{subsec:John}

We first deal with the setting in~\eqref{eq:questamessadopoassunzionesoloperpseudodistanza} and
we obtain the generalizations of Theorems~\ref{thm:stab_pseudodistance} and~\ref{thm:stab_Asymmetry}.
Then, by further assuming \eqref{A4bis} (and hence in the setting \eqref{otto.1}),
we establish the generalizations of Theorem \ref{thm:stability_radii}.
As a consequence, we thus obtain Theorem~\ref{MAIN-DUE}.

\medskip

The formal framework in which we work is the following.
A domain $D$ in $\RR^N$ is a {\it $b_0$-John domain}, with~$b_0 \ge 1$, if each pair of distinct points $a$ and $b$ in $D$ can be joined by a
curve $\ga: \left[0,1 \right] \rightarrow D$ such that
$\ga(0)=a$, $\ga(1)=b$, and
\begin{equation}\label{eq:Johncondition}
\de (\ga(t)) \ge b_0^{-1} \min{ \left\lbrace |\ga(t) - a|, |\ga(t) - b| \right\rbrace  } .
\end{equation}

We emphasize that the class of John domains is huge: it contains Lipschitz domains, but also very irregular domains with fractal boundaries such as, e.g., the Koch snowflake. For more details on John's domains, see~\cite[Section 3.2]{Pog2} or \cites{Ai, MS, NV}, and references therein.
Here, we just notice that, if \eqref{A4} is satisfied then $\Om \setminus \ol{\om}$ is surely a $b_0$-John domain with
$b_0 \le d_\Om /r_i$ (see \cite[(iii) of Remark 3.12]{Pog2}).

As already mentioned in Remark \ref{rem:poincare valgono anche per John}, Lemma~\ref{lem:John-two-inequalities} and Corollary~\ref{cor:JohnPoincareaigradienti} still hold true without the assumption of the uniform interior sphere condition, if $D:=\Om \setminus \ol{ \om}$ is a John domain (see also~\cite[Section 3.2]{Pog2}). In this case, explicit estimates (now depending on the John parameter $b_0$ instead that on $r_i$) of the relevant constants of Lemma~\ref{lem:John-two-inequalities} and Corollary~\ref{cor:JohnPoincareaigradienti} can be found in~\cite[Remark 2.4]{MP3}.
For the reader's convenience we report here the only estimate that we need to conclude our reasoning, that is, 
\begin{equation}\label{eq:stimaesplicitamuperb0John}
\ol{\mu}_{r, p, \al} ( D )^{-1} \le k_{N,\, r, \, p,\, \al} \, b_0^N | D |^{\frac{1-\al}{N} +\frac{1}{r} +\frac{1}{p} } ,
\end{equation}
where $r$, $p$ and~$ \al$ are as in \eqref{eq:conditionHS}.

Now we point out the main changes to perform in this situation
in order to get Theorem \ref{thm:stab_pseudodistance} and its
corresponding consequences in Theorems \ref{JAHS334}
when \eqref{A3} and~\eqref{A4} are dropped and replaced just by \eqref{A3bis}.

We notice that
assumption \eqref{A3} had been exploited only
to allow the use of \eqref{eq:relationdist} in the proof of
Lemma \ref{lem:genericv-trace inequality}. However,
we notice that \eqref{instgr} still holds true without that assumption.

Thus, in order to generalize the stability result of
Theorem~\ref{thm:stab_pseudodistance}, we replace item (ii) of
Lemma~\ref{lem:genericv-trace inequality} with the following result:

\begin{lem}
\label{lem:relaxed assumption_genericv-trace inequality}
Let $\Om \setminus \ol{\om} \subset \RR^N$ be a bounded $b_0$-John domain satisfying \eqref{A3bis}. Let~$u \in C^1 (\ol{\Om} \setminus \om)$ satisfy~\eqref{eq:problem} and \eqref{eq:overdetermination}, and assume that~$u \le 0$ on~$\pa \om$.
Let $v \in C^2 (\ol{\Om} \setminus \om)$ be a harmonic function in $\Om \setminus \ol{\om}$.

If 
\begin{equation}\label{jierety78yndfdjb}
\int_{\Om \setminus \ol{\om} } \na v \, dx = 0,\end{equation}
then it holds that
\begin{eqnarray*}&&
\int_{\Ga} |\na v|^2 \, d\cH^{N-1} \le \frac{2}{c} \left(1+\frac{N}{\ol{\mu}_{2,2, 1 }(\Om \setminus \ol{\om} )^2 } \right)  \int_{\Om \setminus \ol{\om} } (-u) |\na^2 v|^2 dx 
\\&&\qquad\qquad\qquad\qquad-
\frac{1}{c} \int_{\pa^* \om} \left\lbrace |\na v|^2 u_{\nu} - 2 u < \na^2 v \na v , \nu > \right\rbrace \, d\cH^{N-1}.
\end{eqnarray*}
\end{lem}

\begin{proof}
We follow the proof of Lemma \ref{lem:genericv-trace inequality}
until \eqref{fieryu458y45hn}. 
Then, by using \eqref{eq:overdetermination}, formula~\eqref{fieryu458y45hn} becomes
\begin{equation}\label{eq:provalemmarelaxedpseudodistance}\begin{split}
c \, \int_{\Ga} |\na v|^2 \, d\cH^{N-1} \le\;& \int_{\Om \setminus \ol{\om} } |\na v|^2 \,dx +
2 \int_{\Om \setminus \ol{\om} } (-u) |\na^2 v|^2\, dx  
\\&\qquad\qquad-
\int_{\pa^* \om} \left\lbrace |\na v|^2 u_{\nu} - 2 u < \na^2 v \na v , \nu > \right\rbrace \, d\cH^{N-1}.
\end{split}\end{equation}

Now, in light of~\eqref{jierety78yndfdjb} and Remark \ref{rem:poincare valgono anche per John},
we can
use item~(ii) in Corollary~\ref{cor:JohnPoincareaigradienti}, applied here
with~$D := \Om \setminus \ol{\om}$, $r:=p:=2$ and $\al:=1$, and we deduce
from~\eqref{eq:provalemmarelaxedpseudodistance} that
\begin{equation*}\begin{split}
c \, \int_{\Ga} |\na v|^2 \, d\cH^{N-1} \le\;& \frac{1}{ \ol{\mu}_{2,2,1}(\Omega\setminus\overline\omega)^{2}}
\int_{\Om \setminus \ol{\om} }\delta^2 |\na^2 v|^2 \,dx +
2 \int_{\Om \setminus \ol{\om} } (-u) |\na^2 v|^2\, dx  
\\&\qquad\qquad-
\int_{\pa^* \om} \left\lbrace |\na v|^2 u_{\nu} - 2 u < \na^2 v \na v , \nu > \right\rbrace \, d\cH^{N-1}.
\end{split}\end{equation*}
{F}rom this and~\eqref{instgr}, one obtains the desired estimate. 
\end{proof}

We remark that the same generalization could be applied -- using item~(i) in
Corollary~\ref{cor:JohnPoincareaigradienti} -- also to item (i) of
Lemma \ref{lem:genericv-trace inequality}, that could be
useful for other choices of $z$.

\medskip

%
%
%
In
the present subsection
%
%
we maintain the same
choice \eqref{eq:choicez_baricenter} for the point $z$, that is,
\begin{equation}\label{eq:PERIMETROFINITOchoicezbarycenter}
z:= \frac{1}{|\Om \setminus \ol{\om}|} \left\lbrace \int_{\Om \setminus \ol{\om}} x \, dx - N \int_{\pa^* \om} u\;\nu \, d \cH^{N-1} \right\rbrace.
\end{equation}

In this setting, thanks to Lemma~\ref{lem:relaxed assumption_genericv-trace inequality},
Theorem~\ref{thm:stabestimateforpsudodistance} is replaced by
the following statement:

\begin{thm}\label{thm:relaxed_stabestimateforpsudodistance}
Let $\Om \setminus \ol{\om} \subset \RR^N$ be a bounded $b_0$-John domain satisfying \eqref{A3bis}. Let~$u \in C^2 (\ol{\Om} \setminus \om)$ satisfy~\eqref{eq:problem} and \eqref{eq:overdetermination}, and assume that~$u \le 0$ on~$\pa \om$. 

Then, with the notation of~\eqref{quadratic} and \eqref{eq:PERIMETROFINITOchoicezbarycenter},
%
%
we have that
\begin{equation}\begin{split}\label{0895654tfbvnblk}&
\int_{\Ga}  \left| \frac{|x-z|}{N} - c \right|^2 \, d\cH^{N-1} \\ \le& 
\frac{1}{c} \left(1+\frac{N}{ \ol{\mu}_{2,2, 1 }(\Om \setminus \ol{\om} )^2 } \right) \Biggl\{ 
 \int_{ \pa^* \om} \biggl[ c^2  \left( \frac{ < x, \nu>}{N} - u_\nu \right)  +
 2u \left( \frac{< x, \nu >}{N} - u_\nu \right)  
\\&\quad+
 u_\nu | \na u |^2 - 2 \frac{< x, \na u >}{N} u_\nu +
 | \na u|^2 \frac{< x , \nu >}{N} + \frac{2}{N} u u_\nu - 2 <\na^2 u \na u , \nu > u 
 \biggr] \, d\cH^{N-1} \, \Biggr\} 
\\&\quad-
\frac{1}{c} \int_{\pa^* \om} \left\lbrace |\na h|^2 u_{\nu} - 2 u < \na^2 h \na h , \nu > 
\right\rbrace \, d\cH^{N-1}
.\end{split}\end{equation}
\end{thm}

In the same way in which Theorem \ref{thm:stabestimateforpsudodistance}
led to Theorem
\ref{thm:stab_pseudodistance}, now
Theorem \ref{thm:relaxed_stabestimateforpsudodistance} easily leads to
a general stability result for John domains of finite perimeter. In this setting, using Theorem~\ref{thm:relaxed_stabestimateforpsudodistance}
in place of Theorem \ref{thm:stabestimateforpsudodistance}, one obtains
a statement analogous to Theorem \ref{thm:stab_pseudodistance},
with $\pa \om$
replaced with $\pa^* \om$ in \eqref{eq:assumptions_stabgeneral} and
\eqref{eq:assumption_further_stabpseudodistance}, even if assumptions
\eqref{A3} and \eqref{A4} are replaced by \eqref{A3bis}.
The precise result is the following:

\begin{thm}[General stability result for a pseudodistance under relaxed assumptions]
\label{thm:relaxed_stab_pseudodistance}
Let $\Om \setminus \ol{\om}$ a $b_0$-John domain. 
Let $u \in C^2 (\ol{\Om} \setminus \om)$ satisfy \eqref{eq:problem} 
and \eqref{eq:overdetermination},
and suppose that~$u \le 0$ on $\pa \om$.
Let assumption~\eqref{A3bis} be verified, and $z$ be as in \eqref{eq:PERIMETROFINITOchoicezbarycenter}.

If $\psi:[0,\infty)\to[0,\infty)$ is a continuous function vanishing at $0$ such
that
\begin{equation}\label{8567487hjifhew}
\left. \begin{aligned}
\int_{\pa^* \om} (-u) \, d\cH^{N-1} = \int_{\pa^* \om} |u| \, d\cH^{N-1}
\\
\left| \int_{\pa^* \om}  u \, u_\nu \, d\cH^{N-1} \right|
\\
\left| \int_{\pa^* \om}  | \na u|^2 \, u_\nu \, d\cH^{N-1} \right|
\\
\int_{\pa^* \om} | \na u|^2 \, d\cH^{N-1}
\\
\left| \int_{\pa^* \om} < \na^2 u \, \na u , \nu> \, u \, d\cH^{N-1} \right|
\end{aligned}
\right\rbrace
\le \psi(\eta) \quad \text{ with } \eta= \cH^{N-1} (\pa^* \om) \, \text{ or } \, \eta
= \bar d_\om ,
\end{equation}
and
\begin{equation}\label{eq:856richiamo}
\left| \int_{\pa^*\om} u \, < \na^2 u \, (x-z) , \nu > \, d\cH^{N-1} \right| \le \psi(\eta) \quad \text{ with } \eta=
\cH^{N-1} (\pa^* \om) \, \text{ or } \, \eta = \bar d_\om ,
\end{equation}
then
\begin{equation}\label{eq:miserveperultima}
\int_{\Ga}  \left| \frac{|x-z|}{N} - c \right|^2 \, d\cH^{N-1} \le C \, \psi(\eta),
\end{equation}
where $C$ is a constant depending on $N$, $b_0$, $d_\Om$, and $c$.
\end{thm}

\begin{rem}\label{rem:lowerboundcsenzasferainternamaconJohn}
{\rm Concerning the constant in 
formula~\eqref{eq:miserveperultima} of
Theorem~\ref{thm:relaxed_stab_pseudodistance},
we observe that the dependence from~$b_0$
comes from the estimate
of~$\ol{\mu}_{2,2,1}(\Om \setminus \ol{\om} )$
in~\eqref{eq:stimaesplicitamuperb0John}.
We also remark that the volume appearing in \eqref{eq:stimaesplicitamuperb0John} can be estimated from above in terms of $d_\Om$, by means
of~\eqref{eq:nuovadarichiamareinsiemeainequality_volume_diameter} and~\eqref{eq:inequality_volume_diameter}. 

The dependence of $C$ in~\eqref{eq:miserveperultima}
on $c$ comes from the fact that
the quantity~$1/c$ plays a role in the estimates,
as it can be seen from
formula~\eqref{0895654tfbvnblk} in
Theorem~\ref{thm:relaxed_stabestimateforpsudodistance}. 

We point out that such a dependence can in fact
be replaced with the dependence on $| \Ga|$, by means of a suitable lower bound for $c$.
Namely, we claim that
\begin{equation}\label{j854yt48ygnhknfkjg}
c\ge
k_N \frac{1}{| \Ga|} \left( \frac{d_\Om}{ b_0 } \right)^N,
\end{equation}
if $\eta$ is small enough, where~$k_N$
is a positive constant depending on~$N$.
To prove it, we first observe that
\begin{equation}\label{do048t9547yh}
{\mbox{for any $a$, $b \in \Om \setminus \ol{\om}$, a ball of radius $\frac{|a-b|}{2 b_0}$ is contained in $\Om \setminus \ol{\om}$.}}
\end{equation}
Indeed, since $\Om \setminus \ol{\om}$
is a~$b_0$-John domain, we have that
there exists a curve~$\gamma:[0,1]\to \Om \setminus \ol{\om}$
such that~$\gamma(0)=a$, $\gamma(1)=b$
and~\eqref{eq:Johncondition} holds true. Moreover, by inspection
one sees that
there exists~$t^*\in(0,1)$ such
that~$|\gamma(t^*)-a|=|\gamma(t^*)-b|$, and so, by~\eqref{eq:Johncondition},
\begin{equation*}
\delta(\gamma(t^*))\ge b_0^{-1}|\gamma(t^*)-a|,\end{equation*}
which implies that
\begin{equation}\label{94tnvfnn}
B_{b_0^{-1}|\gamma(t^*)-a|}(\gamma(t^*))
\subset \Om \setminus \ol{\om}.\end{equation}
Furthermore, by the triangle inequality,
\begin{equation*}
|a-b|\le |\gamma(t^*)-a|+|\gamma(t^*)-b|= 2 |\gamma(t^*)-a|.
\end{equation*}
This and~\eqref{94tnvfnn} imply~\eqref{do048t9547yh}.

As a consequence of~\eqref{do048t9547yh}, we have that,
for any $a$, $b \in \Om \setminus \ol{\om}$,
\begin{equation}\label{kfoy965y}
|\Om \setminus \ol{\om}| \ge |B_1| \left( \frac{|a-b|}{2 b_0} \right)^N .
\end{equation}
Now we choose~$a$ and~$b$ in such a way that $|a-b|$ is arbitrarily close to $d_\Om$, say $|a-b| \ge d_\Om / 2$, and we obtain
from~\eqref{kfoy965y} that
\begin{equation}\label{eq:stimaJohninradius}
|\Om| - |\om| \ge |B_1| \left( \frac{d_\Om}{4 b_0} \right)^N.
\end{equation} 
Then,
by exploiting ~\eqref{eq:value_c},
\eqref{8567487hjifhew} and~\eqref{eq:stimaJohninradius}, we see that
$$c|\Gamma|=|\Om|-|\om|-\int_{\partial^* \omega}u_\nu \, d\cH^{N-1} \ge
\frac{|\Om|-|\om|}2\ge 
\frac{|B_1|}2 \left( \frac{d_\Om}{4 b_0} \right)^N
,
$$
as long as~$\eta$ is sufficiently small.
This proves~\eqref{j854yt48ygnhknfkjg}.
}
\end{rem}

Theorem~\ref{thm:relaxed_stab_pseudodistance}
also leads to a stability bound for the asymmetry defined in \eqref{eq:asymmetryallaFraenkel} by means of the following generalization of
Lemma \ref{lem:relationasymmetrypseudodistance} to the case of John domains:

\begin{lem}\label{lem:JOHN_relationasymmetrypseudodistance}
Let $\Om \subset \RR^N$ be a bounded $b_0$-John domain with Lipschitz boundary $\Ga$.
%
%
Then, there exists a positive constant $C$ only depending on $N$, $b_0$, $c$, such that
\begin{equation*}
\frac{| \Om \De B_{Nc}(z) |}{| B_{Nc}(z) |} \le C \left[ \int_{\Ga} \left| \frac{|x-z|}{N} - c \right|^2 \, d\cH^{N-1} \right]^{\frac{1}{2}} .
\end{equation*}
\end{lem}

\begin{proof}
The desired result follows by applying \cite[Lemma 11]{Fe} with 
$$K:=\max \left\lbrace \frac{4 b_0 Nc}{d_\Om} , \, \left( \frac{d_\Om}{2 Nc} \right)^N \right\rbrace \quad \text{ and } \quad r:=Nc .$$ 
Notice that \cite[Lemma 11]{Fe} can be applied with these choices for $K$ and $r$. Indeed,
formula~\eqref{eq:stima_asimmetria_volume_riserve} is still
satisfied.
On the other hand, to deduce
formula~\eqref{5.3bis} in this setting, we observe that
\begin{equation}\label{948t94y5uy9}
r_{in} (\Om) \ge \frac{d_\Om}{4 b_0} ,
\end{equation}
where $r_{in} (\Om):= \max_{\ol{\Om}} \de_\Ga (x)$ denotes the {\it inradius} of $\Om$.
To prove~\eqref{948t94y5uy9}, one can use~\eqref{do048t9547yh}
and choose $a$ and $b$ in $\Om$ such that $|a-b| \ge d_\Om / 2$.

Then, from~\eqref{948t94y5uy9} one deduce that
$$
K r_{in}(\Om) \ge Nc \, \frac{4 b_0 r_{in} (\Om)}{d_\Om} \ge Nc.
\qedhere
$$
\end{proof}

In light of Lemma~\ref{lem:JOHN_relationasymmetrypseudodistance},
we also deduce from Theorem~\ref{thm:relaxed_stab_pseudodistance}
a stability result for an asymmetry in this setting:

\begin{thm}[General stability result for an asymmetry under relaxed assumptions]\label{thm:relaxed_stab_Asymmetry}
Let $\Om \setminus \ol{\om}$ a $b_0$-John domain. 
Let $u \in C^2 (\ol{\Om} \setminus \om)$ satisfy \eqref{eq:problem} 
and \eqref{eq:overdetermination},
and suppose that~$u \le 0$ on $\pa \om$.
Let assumption~\eqref{A3bis} be verified, and $z$ be as in \eqref{eq:PERIMETROFINITOchoicezbarycenter}.

Let $\psi:[0,\infty)\to[0,\infty)$ be a continuous function vanishing at $0$ such
that \eqref{8567487hjifhew} and \eqref{eq:856richiamo} hold true.
%
%
%

Then, 
%
it holds that
\begin{equation}\label{eq:relaxed_dithmstab_asymmetry}
\frac{| \Om \De B_{Nc}(z) |}{| B_{Nc}(z) |} \le C \, \psi(\eta)^{1/2},
\end{equation}
where $C$ is a constant depending on $N$, $b_0$, $d_\Om$, and $c$.
\end{thm}

\begin{rem}
{\rm The dependence of the constant $C$
in~\eqref{eq:relaxed_dithmstab_asymmetry}
on $c$ can be replaced with the dependence on $| \Ga|$, when $\eta$ is small enough. In fact, a lower bound for $c$ in terms of $N$, $b_0$, $d_\Om$, $1/|\Ga|$ has been obtained in \eqref{j854yt48ygnhknfkjg};
%
%
%
the upper bound for $c$ in terms of $N$ and $d_\Om$ obtained in \eqref{eq:STARSTARSTARSTAR}
%
%
still holds true.
}
\end{rem}

Our next objective is to provide a stability
estimate for $\rho_e - \rho_i$,
as given in formula~\eqref{eq:generalstability},
in the more general framework of John domains.
This will lead to a general version
of Theorem~\ref{thm:stability_radii}, which
in turn will produce a general version of
Theorem~\ref{JAHS334}.

Notice that, when assumption \eqref{A3} is replaced by \eqref{A3bis},
we have to use \eqref{instgr} instead of \eqref{eq:relationdist} to bound from below the left-hand side of \eqref{eq:FIdconsovradeterminazione}.
Thus, in this case, the quantity that has to be put in relation with $\rho_e -\rho_i$ is $\nr \de \, \na^2 h  \nr_{2,\Om \setminus \ol{ \om } }$ instead of $\nr \de^{\frac{1}{2}} \, \na^2 h  \nr_{2,\Om \setminus \ol{ \om } }$.
For this reason, the exponents $\tau_N$ in Theorem~\ref{thm:Johnrelaxed_radiistability}, stated next, 
become worse with respect to those of Theorem \ref{thm:stability_radii}.

More precisely, the counterpart of Theorem~\ref{thm:serrin-W22-stability} in this more general setting is the following:

\begin{thm}\label{48yujewd74756}
\label{thm:relaxedJohn_serrin-W22-stability} 
Let $\Om \setminus \ol{\om} \subset\RR^N$ be a bounded $b_0$-John domain,
satisfying~\eqref{A4bis}, and suppose that~$\Ga$
is of class~$C^1$.
Let~$u$ 
%
%
satisfy \eqref{eq:problem}, $u \in C^1 \left( \left( \Om \setminus \ol{\om} \right) \cup \Ga \right)$, and suppose that~$u \le 0$ on $\pa \om$.
Let~$q$
be as in \eqref{quadratic} with $z$ chosen as in \eqref{eq:PERIMETROFINITOchoicezbarycenter},
and assume that~$z$ belongs to $\Om$.
Let~$h$ be as in~\eqref{DEh}.

Then, there exists a positive constant $C$ such that
\begin{equation}\label{eq:esponentipeggio}
\rho_e-\rho_i\le C\, \nr \de \, \na^2 h  \nr_{2,\Om \setminus \ol{ \om } }^{\tau_N}  \quad \text{with} \quad 
\tau_N=
\begin{cases}
1- \theta, \text{ for any } \theta>0 , & \text{ when } N=2 \\
2/N , & \text{ when } N \ge 3 .
\end{cases}
\end{equation}

The constant $C$
depends on $N$, $r_i$, $b_0$, $d_\Om$, $M$ (as
defined in \eqref{bound-gradient}).
\end{thm}
\begin{proof}
The desired estimate easily follows by reasoning as in the proof of items (ii) and (iii) of Theorem~\ref{thm:serrin-W22-stability}. The only difference is that now we apply Poincar\'e inequalities of item (ii) of Corollary \ref{cor:JohnPoincareaigradienti} with $\al:=1$ (instead that $1/2$).

More precisely, when $N=2$ we apply, with $v:=h$ and $D:=\Om \setminus \ol{\om}$, item (ii) of Corollary \ref{cor:JohnPoincareaigradienti} with $r:=2(1-\theta)$, $p:=2$, $\al:=1$ and \eqref{John-harmonic-poincare} with $r:= 2(1 - \theta)/ \theta$, $p:= 2 (1- \theta)$, $\al:=0$.
When $N \ge 3$ we apply, with $v:=h$ and $D:=\Om \setminus \ol{\om}$, item (ii) of Corollary \ref{cor:JohnPoincareaigradienti} with $r:= p :=2$, $\al:=1$ and \eqref{John-harmonic-poincare} with  $r:= 2N/(N-2)$, $p:=2$, $\al:=1$.

We stress that assumption \eqref{A4} in items (ii) and (iii) of Theorem \ref{thm:serrin-W22-stability} was assumed only to guarantee that the Poincar\'e inequalities of Lemma~\ref{lem:John-two-inequalities} and Corollary~\ref{cor:JohnPoincareaigradienti} could be applied with $D:=\Om \setminus \ol{\om}$; in fact, notice that Lemma \ref{lem:L2-estimate-oscillation}, that was the main ingredient of those proofs, was already stated under the weaker assumption \eqref{A4bis}. Being now $\Om \setminus \ol{\om}$ a $b_0$-John domain, in light of Remark \ref{rem:poincare valgono anche per John} and \eqref{eq:stimaesplicitamuperb0John} it is clear that we can still apply those Poincar\'e inequalities 
with $D:=\Om \setminus \ol{\om}$, even if \eqref{A4} has been dropped.
\end{proof}

By applying Theorem~\ref{48yujewd74756}
in the place of
Theorem~\ref{thm:serrin-W22-stability}, we thus derive a counterpart
of Theorem \ref{thm:stability_radii} for John domains.
Namely,
in this setting, the result in
Theorem \ref{thm:stability_radii} still holds true if
assumptions \eqref{A3}, \eqref{A4} are
replaced with \eqref{A3bis}, \eqref{A4bis}
-- with $\pa \om$
replaced with $\pa^* \om$ in \eqref{eq:assumptions_stabgeneral} 
-- but in this situation
the exponents $\tau_N$ are those in \eqref{eq:esponentipeggio}.
The precise result is indeed the following one:

\begin{thm}\label{thm:Johnrelaxed_radiistability}
Let $\Om \setminus \ol{\om} \subset\RR^N$ be a bounded $b_0$-John domain.
Let~$u \in C^2 (\ol{\Om} \setminus \om)$ 
satisfy \eqref{eq:problem}
and \eqref{eq:overdetermination}, and
suppose that~$u \le 0$ on $\pa \om$.
Let assumptions~\eqref{A3bis} and~\eqref{A4bis} be verified.

Assume also that the point $z$ chosen in~\eqref{eq:PERIMETROFINITOchoicezbarycenter} belongs to~$\Om$.

If $\psi:[0,\infty)\to[0,\infty)$ is a continuous function vanishing at $0$ such that \eqref{8567487hjifhew} holds true,
%
%
then
\begin{equation}\label{eq:Johnraddistabilityprova}
\rho_e-\rho_i\le C \, \psi(\eta)^{\tau_N/2},
\end{equation}
where $\tau_N$ is as in \eqref{eq:esponentipeggio} and $C$ is a positive
constant that depends on $N$, $r_i$, $b_0$, $d_\Om$, $M$ (as
defined in \eqref{bound-gradient}).
\end{thm}

\begin{proof} 
The desired result easily follows by putting together 
Theorem~\ref{48yujewd74756} and formulas~\eqref{instgr}, \eqref{eq:CSintermsofh} and~\eqref{eq:FIdconsovradeterminazione}.
\end{proof}
 
As can be deduced from the discussion before Theorem \ref{48yujewd74756}, if~\eqref{A3} holds true,
then the statement in Theorem~\ref{thm:Johnrelaxed_radiistability} can be strengthen,
since in this case one can obtain~$\tau_N$ as in Theorem~\ref{thm:serrin-W22-stability}
(we do not enter into these details, since this special statement will not
be used in what follows).
%
%
\medskip

The corresponding generalization of Theorem~\ref{JAHS334} easily follows
from Theorems \ref{thm:relaxed_stab_pseudodistance}, \ref{thm:relaxed_stab_Asymmetry},
and~\ref{thm:Johnrelaxed_radiistability}, 
and it can be stated as follows:

\begin{thm}\label{J:l90k1234cniw}
Let $\Om \setminus \ol{\om} \subset\RR^N$ be a bounded $b_0$-John domain.
Let $u \in C^2 (\ol{\Om} \setminus \om)$ satisfy \eqref{eq:problem}
and \eqref{eq:overdetermination},
and suppose that~$u \le 0$ on $\pa \om$.
Let assumption \eqref{A3bis} be verified,
and~$z$ be as in~\eqref{eq:PERIMETROFINITOchoicezbarycenter}.

If there exists $K >0$ such that 
\begin{equation}\label{eq:ecchec}
\nr u \nr_{C^2 (\pa \om)} \le K ,
\end{equation}
then
\begin{equation}\label{eq:generapseudodistC2normajohndomain}
\int_{\Ga}  \left| \frac{|x-z|}{N} - c \right|^2 \, d\cH^{N-1} \le C {\mathcal{H}}^{N-1}(\pa^* \om ),
\end{equation}
where $C$ is a positive constant depending on 
$N$, $b_0$, $d_\Om$, $c$ and $K$.

Also, 
%
%
it holds that
\begin{equation}\label{eq:relaxed_asymmetrystabconnormaC^2}
\frac{| \Om \De B_{Nc} (z) |}{|B_{Nc} (z)|} \le C {\mathcal{H}}^{N-1}(\pa^* \om )^{1/2},
\end{equation}
where $C$ is a positive constant depending on $N$, $b_0$, $d_\Om$, $c$, and $K$.

If in addition \eqref{A4bis} is verified and $z \in \Om$, we have that
\begin{equation}\label{eq:JohnstabconC2norma}
\rho_e-\rho_i\le C \big({\mathcal{H}}^{N-1}(\pa^* \om)\big)^{\tau_N/2},
\end{equation}
where $\tau_N$ are as in~\eqref{eq:esponentipeggio},
and $C$ is a positive
constant depending on~$N$, $r_i$, $b_0$, $d_\Om$, and $K$.
\end{thm}

The proof of Theorem~\ref{MAIN-DUE}
now plainly follows from Theorem~\ref{J:l90k1234cniw}.

\begin{rem}\label{rem:PROVA TOGLIERE C CON CONDIZIONE SU PERIMETRO}
{\rm 
As in Theorem \ref{JAHS334}, the constant $C$ in \eqref{eq:JohnstabconC2norma} does not depend
on $M$ (as defined in~\eqref{bound-gradient}), differently from that appearing in \eqref{eq:Johnraddistabilityprova}.

The dependence of the constants $C$ on $c$ in \eqref{eq:generapseudodistC2normajohndomain} and \eqref{eq:relaxed_asymmetrystabconnormaC^2} could be replaced with the dependence on $| \Ga |$, at least when ${\mathcal{H}}^{N-1}(\pa^* \om)$ is small enough. In fact, when, e.g.,
\begin{equation}\label{eq:assumperimetroom}
{\mathcal{H}}^{N-1}(\pa^* \om) < \frac{1}{2K} \left(\frac{d_\Om}{4 b_0} \right)^N ,
\end{equation}
by exploiting \eqref{eq:value_c}, \eqref{eq:ecchec} and \eqref{eq:stimaJohninradius} we have that
\begin{eqnarray*}&&
c |\Ga| = | \Om| - | \om| - \int_{ \pa^* \om} u_\nu \, d \cH^{N-1} \ge
| \Om| - | \om| -K{\mathcal{H}}^{N-1}(\pa^* \om)\\&&\qquad
\ge |B_1| \left( \frac{d_\Om}{4 b_0} \right)^N
-K{\mathcal{H}}^{N-1}(\pa^* \om)\ge
\frac{|B_1|}2 \left( \frac{d_\Om}{4 b_0} \right)^N .
\end{eqnarray*}
As a result,
$$
c \ge \frac{|B_1|}{2|\Gamma|} \left( \frac{d_\Om}{4 b_0} \right)^N
.
$$

On the other hand, in this case, by \eqref{eq:value_c}, \eqref{eq:disuguaglianze per upper bound c}, and \eqref{eq:assumperimetroom}, we also get a suitable upper bound for $c$, that is
\begin{eqnarray*}&&
c \le \frac{1}{|\Gamma|}
\left(| \Om| - | \om| - \int_{ \pa^* \om} u_\nu \, d \cH^{N-1}  \right)
\le 
\frac{d_\Om}{2N} +\frac1{|\Gamma|}K{\mathcal{H}}^{N-1}(\pa^* \om) 
\le 
\frac{d_\Om}{2N} +\frac1{2|\Gamma|}\left(\frac{d_\Om}{4 b_0} \right)^N
.
\end{eqnarray*}
}
\end{rem}

\subsection{A new different choice for \texorpdfstring{$z$}{z} and proof of Theorem \ref{MAIN-TRE}}\label{subsec:differentchoicesofz}

Our next goal is to show that
it is possible to obtain Theorem~\ref{thm:stability_radii}
and its consequences in Theorem \ref{JAHS334},
with $\tau_N$ given in Theorem~\ref{thm:serrin-W22-stability}
%
%
and in the more general setting \eqref{otto.2}, provided that
we make a different choice of $z$. 

The main difficulty in this setting is
that no regularity information is
available on $\pa \om$ (not even being $\Om \setminus \ol{ \om}$ a John domain),
and therefore we cannot apply Poincar\'e inequalities on all $\Om \setminus \ol{\om}$,
making it difficult
to establish an appropriate variant of Theorem \ref{thm:serrin-W22-stability}.

The key idea to overcome this difficulty
is to perform the necessary
Poincar\'e inequalities on a suitable subset of $\Om \setminus \ol{\om}$. Such a suitable subset is $\Om^c_{r_i}$, where we are using the notation introduced in \eqref{A4bis} and \eqref{def:complementary parallel set}. 
Notice that, by~\eqref{A4bis}, it holds that
$$
\Om_{r_i}^c \subset \Om \setminus \ol{\om} .
$$

With this setting, we start with the following estimate:

\begin{lem}
\label{lem:application-refined-Lp-estimate-oscillation-generic-v}
Let $\Om \setminus \ol{\om} \subset\RR^N$
satisfy \eqref{A4bis}, and suppose that~$\Ga$
is of class~$C^1$.
Let $v$ be a harmonic function in $\Om \setminus \ol{\om} $ of
class $C^1 (\ol{\Om^c_{r_i /2}})$, and let~$G$ be an upper bound for the gradient of $v$ on $\ol{\Om^c_{r_i/2}}$.
\par
Then, given~$p\ge1$, there exist two positive
constants $\tilde{a}_{N,p}$ and $\tilde{\al}_{N,p}$ depending only on $N$ and $p$ such that if
\begin{equation}
\label{application-refined-smallness-generic-v}
\nr v - v_{\Om^c_{r_i} } \nr_{p, \Om^c_{r_i} } \le \tilde{\al}_{N,p} \, r_{i}^{\frac{N+p}{p}}  G  ,
\end{equation}
then we have that
\begin{equation}
\label{application-refined-Lp-stability-generic-v}  
\max_{\Ga} v - \min_{\Ga} v \le \tilde{a}_{N,p} \,  G^{ \frac{N}{N+p} } \, \nr v - v_{\Om^c_{r_i} } \nr_{p, \Om^c_{r_i} }^{ p/(N+p) }.
\end{equation} 
\end{lem}
\begin{proof}
The desired result follows by applying Lemma \ref{lem:refined_Lp-estimate-oscillation-generic-v} with $r_i := r_i/2$ and $\la := v_{\Om^c_{r_i}}$. 
Indeed, since
$$
B_{r_i /2} (x_0) \subset \Om^c_{r_i} ,
$$
it holds that
$$
\nr v - v_{\Om^c_{r_i}} \nr_{p, B_{r_i /2}(x_0) }\le \nr v - v_{\Om^c_{r_i}} \nr_{p, \Om_{r_i}^c } ,
$$
and \eqref{application-refined-smallness-generic-v}, \eqref{application-refined-Lp-stability-generic-v} follow from \eqref{eq:suppalla_refined-smallness-generic-v}, \eqref{eq:suppalla_refined-Lp-stability-generic-v}.
\end{proof}

We recall that, thanks to \eqref{A4bis},
$\Ga_{r_i}$ inherits the same regularity of $\Ga$. More precisely, 
%
%
we have that:
\begin{equation}\label{eq:tubneighbregularity}
\Ga \, \text{ is } \, C^k \quad \Longrightarrow \quad  \Ga_{r_i} \, \text{ is } \, C^k , \quad \text{for } k \ge 1.
\end{equation}
This fact relies on the regularity of the distance function. The case $k \ge 2$ has been proved in Appendix of \cite{GT} (see also \cite[Theorem 3]{KrantzParks-distance}).
The case $k=1$ can be deduced from \cite[Theorem 2]{KrantzParks-distance}, but we do not need this refinement here. 
%
%

Moreover, we have that:
\begin{lem}\label{lem:boprovoachiamarlosferainternapertubulare}
Assume that \eqref{A4bis} holds true. Then, the domain $\Om^c_{r_i}$ in~\eqref{def:complementary parallel set}
satisfies the uniform interior sphere condition with radius $r_i /2$.
\end{lem}
\begin{proof}
%
%
%
For any $y \in \Ga_{r_i}$ and $x \in \Ga$  such that
$y= x - r_i \nu(x)$, by \eqref{A4bis} and definition \eqref{def:complementary parallel set}, we have that
$B_{r_i /2}\left(\frac{x+y}{2}\right) \subset  \Om^c_{r_i}$.
It follows that~$B_{r_i /2}(\frac{x+y}{2})$ is an interior touching ball in $\Om^c_{r_i}$ (at $x$ and $y$).
\end{proof}

In order to use the Poincar\'e inequality of item (ii) of Corollary~\ref{cor:JohnPoincareaigradienti} (with $v=h$ and $D= \Om^c_{r_i} $), we have to make a new appropriate choice of $z$.

To this aim, a possible choice of $z$ is:
\begin{equation}\label{eq:choicezperPoincaresutubularneighborood}
z := \frac{1}{|\Om^c_{r_i} |} \left\lbrace \int_{\Om^c_{r_i}} x \, dx - N \int_{\Ga_{r_i}} u \nu 
\, d\cH^{N-1} \right\rbrace  ,
\end{equation}
where
\begin{equation}\label{GARO}
\Ga_{r_i} := \{ y\in\Om: \de_\Ga (y) = r_i \} .
\end{equation}
We remark that, with the choice in~\eqref{eq:choicezperPoincaresutubularneighborood},
by \eqref{u9iwhschhss} and Green's identity
it follows that
\begin{equation}\label{eq:choicezperPoincaresutubularneighborood:SOTTO}
\int_{\Om^c_{r_i}} \na h \, dx = 0   .
\end{equation}

We are now in position to prove the counterpart of
Theorem~\ref{thm:serrin-W22-stability} in the present setting.
%
%
%
\begin{thm}
\label{thm:Tubularneighborood-serrin-W22-stability} 
Let $\Om \setminus \ol{\om} \subset\RR^N$ satisfy \eqref{A4bis}, and suppose that~$\Ga$
is of class~$C^1$.
Let~$u$ satisfy \eqref{eq:problem}, $u \in C^1 \left( \left( \Om \setminus \ol{\om} \right) \cup \Ga \right)$, and suppose that~$u \le 0$ on $\pa \om$.
Let~$q$
be as in \eqref{quadratic} with $z$ chosen as in \eqref{eq:choicezperPoincaresutubularneighborood},
and assume that~$z$ belongs to $\Om$.
Let~$h$ be as in~\eqref{DEh}.

Then, there exists a positive constant $C$ such that
\begin{equation}\label{eq:tubular-W2,2_modificato}
\rho_e-\rho_i\le C\, \nr \dist \left( x, \pa \Om^c_{r_i} \right)^{\frac{1}{2} } \, \na^2 h  \nr_{2, \Om^c_{r_i} }^{\tau_N} ,
\end{equation}
with the following specifications:
\begin{enumerate}[(i)]
\item $\tau_2 = 1$;
\item $\tau_3$ is arbitrarily close to~$1$, in the sense that, for any $\theta\in(0,1)$ sufficiently
small, there exists a positive constant $C$ such that  \eqref{eq:tubular-W2,2_modificato} holds with $\tau_3 = 1- \theta$;
\item $\tau_N = 2/(N-1)$ for $N \ge 4$.
\end{enumerate}

The constant $C$
depends on $N$, $r_i$, $d_\Om$, $\max_{\ol{\Om^c_{r_i /2}}} |\na u|$,
and $\theta$ (the latter, only in the case $N=3$).
\end{thm}

\begin{proof}
By using Lemma~\ref{lem:application-refined-Lp-estimate-oscillation-generic-v}
with $v=h$ and reasoning as in Lemma \ref{lem:L2-estimate-oscillation} we see that
$$\rho_e-\rho_i \le C \, \nr h - h_{\Om^c_{r_i}} \nr_{p, \Om^c_{r_i} }^{ p/(N+p) }.$$

Then, we modify the proof of Theorem~\ref{thm:serrin-W22-stability} by using the Poincar\'e inequalities of Lemma~\ref{lem:John-two-inequalities}
and Corollary~\ref{cor:JohnPoincareaigradienti} with $v=h$ on $D= \Om^c_{r_i}$ (instead of taking~$D=\Om \setminus \ol{\om}$).

By 
%
%
Lemma \ref{lem:boprovoachiamarlosferainternapertubulare} we have that the domain $\Om^c_{r_i}$ satisfies the uniform interior sphere condition with radius $r_i/2$, and hence Lemma~\ref{lem:John-two-inequalities}
and Corollary~\ref{cor:JohnPoincareaigradienti} can be applied with $D= \Om^c_{r_i}$.
Also, in light of~\eqref{eq:choicezperPoincaresutubularneighborood:SOTTO},
we have that the choice in~\eqref{eq:choicezperPoincaresutubularneighborood} for $z$ guarantees
that the Poincar\'e inequalities of Corollary~\ref{cor:JohnPoincareaigradienti} can be applied with $v=h$ and $D= \Om^c_{r_i}$.

With these modifications and proceeding as in the proof of Theorem~\ref{thm:serrin-W22-stability}, instead of \eqref{eq:C-provastab-serrin-W22} we obtain the refined estimate \eqref{eq:tubular-W2,2_modificato}
(in the case $N=2$, also
the Sobolev inequality \eqref{eq:immersionSerrinN2VERSIONENEW} has to
be performed here with $\Om \setminus \ol{ \om}$ replaced by $\Om^c_{r_i}$).
%
%
\end{proof}

{F}rom the previous work, we can now obtain
the counterpart of Theorem~\ref{thm:stability_radii} under the relaxed
assumptions~\eqref{A3bis}, \eqref{A4bis}:

\begin{thm}[General stability result for $\rho_e - \rho_i$ under relaxed assumptions]
\label{thm:tubular.relaxed_stability_radii}
Let $u \in C^2 (\ol{\Om} \setminus \om)$ satisfy \eqref{eq:problem} and \eqref{eq:overdetermination}, and
suppose that~$u \le 0$ on $\pa \om$.
Let assumptions~\eqref{A3bis} and~\eqref{A4bis} be verified.
Assume also that the point $z$ chosen as in~\eqref{eq:choicezperPoincaresutubularneighborood} belongs to~$\Om$.

If $\psi:[0,\infty)\to[0,\infty)$ is a continuous function vanishing at $0$ and satisfying \eqref{8567487hjifhew},
%
%
then
\begin{equation*}
\rho_e-\rho_i\le C \, \psi(\eta)^{\tau_N/2},
\end{equation*}
with $\tau_N$ as in Theorem~\ref{thm:Tubularneighborood-serrin-W22-stability}
and $C$ depending on $N$, $r_i$, $d_\Om$, and~$\max_{\ol{\Om^c_{r_i /2}}} |\na u|$.
\end{thm}

\begin{proof}
In the notation of~\eqref{GARO} we have that~$\pa \Om^c_{r_i}= \Ga \cup \Ga_{r_i}$, and, by recalling \eqref{noWea} and \eqref{eq:tubneighbregularity},
%
%
$\Om^c_{r_i}$ is of class $C^2$. 
Moreover, since $\Om^c_{r_i}$ satisfies the uniform interior
sphere condition with radius $r_i/2$,
Lemma~\ref{lem:relationdist} can be applied to~$\Om^c_{r_i}$
in place of~$\Omega\setminus\overline\omega$.
Hence, in this setting, formula~\eqref{eq:relationdist}
can be rephrased as
\begin{equation}\label{BYUS}
-u(x) \ge \frac{r_i}{4 N} \dist \left( x, \pa \Om^c_{r_i} \right) \quad \text{for any} \quad x \in \Om^c_{r_i}.
\end{equation}
Moreover, by the maximum principle,
\begin{equation}\label{BYUS2} -u \ge 0 \quad \text{ on } \quad \Om \setminus \ol{ \om } .\end{equation} 
Hence, since
$$ \Om^c_{r_i} \subset \Om \setminus \ol{\om} ,$$
we deduce from~\eqref{eq:tubular-W2,2_modificato},
\eqref{BYUS} and~\eqref{BYUS2} that
\begin{eqnarray*}&&
\rho_e-\rho_i
\le C\, \left( \int_{\Om^c_{r_i} } \dist \left( x, \pa \Om^c_{r_i} \right)\, |\na^2 h|^2 \, dx \right)^{ \tau_N / 2 } 
\\&&\qquad\le C\, \left( \int_{\Om^c_{r_i} } (-u) |\na^2 h|^2 \, dx \right)^{ \tau_N / 2 } 
\le C\, \left( \int_{\Om \setminus \ol{\om} } (-u) |\na^2 h|^2 \, dx \right)^{ \tau_N / 2 } .
\end{eqnarray*}
Consequently, the desired result follows from \eqref{eq:FIdconsovradeterminazione} (with $\pa \om$ replaced by $\pa^*\om$) and \eqref{eq:CSintermsofh}.
\end{proof}

As a consequence of Theorem \ref{thm:tubular.relaxed_stability_radii}, we thus have the following generalization of \eqref{eq:stabconnormaC^2}:
\begin{thm}\label{IUhtkS45}
Let $u \in C^2 (\ol{\Om} \setminus \om)$ satisfy \eqref{eq:problem} and \eqref{eq:overdetermination},
and suppose that~$u \le 0$ on $\pa \om$.
Let assumptions~\eqref{A3bis} and~\eqref{A4bis} be verified,
and~$z$ be as in~\eqref{eq:choicezperPoincaresutubularneighborood}.

Assume that there exists $K >0$ such that 
\begin{equation*}
\nr u \nr_{C^2 (\pa \om)} \le K ,
\end{equation*}
and that $z \in \Om$.

Then,
\begin{equation*}
\rho_e-\rho_i\le C \,\big({\mathcal{H}}^{N-1}(\pa^* \om )\big)^{\tau_N/2},
\end{equation*}
where $\tau_N$ are as in Theorem~\ref{thm:Tubularneighborood-serrin-W22-stability} and $C$ is a positive
constant depending on $N$, $r_i$, $d_\Om$, and~$K$.
\end{thm}

The proof of Theorem~\ref{MAIN-TRE} is now a plain consequence of
Theorem~\ref{IUhtkS45}.

Of course, from Theorems \ref{thm:relaxed_stab_pseudodistance},
\ref{thm:relaxed_stab_Asymmetry}, \ref{thm:Johnrelaxed_radiistability}, and \ref{thm:tubular.relaxed_stability_radii}, we can also deduce the corresponding generalizations of Theorem~\ref{thm:blowup}.
For instance, from Theorem \ref{thm:tubular.relaxed_stability_radii} we deduce:

\begin{thm}
Let $u \in C^2 (\ol{\Om} \setminus \om)$
satisfy \eqref{eq:problem} and \eqref{eq:overdetermination},
and suppose that~$u \le 0$ on $\pa \om$.
Let assumption \eqref{A4bis} be verified, 
and~$z$ be as in~\eqref{eq:choicezperPoincaresutubularneighborood}.  
Let $\om$ be the union of finitely many
disjoint balls
of radius~$\ve$ and assume that 
$$
\|u\|_{L^\infty(\pa\om)}+\ve\|\nabla u\|_{L^\infty(\pa\om)}+
\ve^2\|\nabla^2 u\|_{L^\infty(\pa\om)} = o(\ve^{ \frac{4-N}{3} })
.$$
Suppose that~$z\in \Om$. Then~$
\rho_e-\rho_i$ is as small as we wish for small~$\ve$.
\end{thm}

\begin{appendix}

\section{Motivation from an optimal heating problem (with possible malfunctioning of the source)}\label{APP:he}

In this section we briefly recall how the simple model
from optimal heating described in~\eqref{0-0-2} 
directly produces the overdetermined condition in~\eqref{0-0-3}.
For this, we consider a divergence free vector field~$v$.
Also, for small~$t\ge0$, we introduce the diffeomorphism given by
$$ \Phi^t(x):=x+tv(x).$$
We set~$\Omega^t:=\Phi^t(\Omega)$ and, given a source~$f\ge0$, we let~$u^t$ be the solution of
$$ \begin{cases}
\Delta u^t =f &{\mbox{ in }}\Omega^t,\\
u^t=0 &{\mbox{ on }}\partial\Omega^t.
\end{cases}$$
We consider the energy functional
$$ I(t):=\frac12\int_{\Omega^t} |\nabla u^t(x)|^2\,dx.$$
We also define
$$ \psi(x,t):=
\frac12 |\nabla u^t(x)|^2.$$
In this way, we have that
\begin{equation*}
\partial_t\psi(x,t)=\nabla u^t(x)\cdot\nabla \partial_t u^t(x)
\end{equation*}
and
$$ I(t)=\int_{\Omega^t} \psi(x,t)\,dx.$$
By the Hadamard's Differentiation Formula (see Theorem~5.2.2
in~\cite{HP}), we know that
\begin{eqnarray*}
I'(0)&=&\int_\Omega\partial_t\psi(x,0)\,dx+
\int_{\partial\Omega} \psi(x,0)\,<\nu(x),v(x)>\,dS_x\\
&=& \int_\Omega
\nabla u^0(x)\cdot\nabla \partial_t u^0(x)\,dx+
\frac12 \int_{\partial\Omega} |\nabla u^0(x)|^2\,<\nu(x),v(x)>\,dS_x.
\end{eqnarray*}
Since~$\Phi^t(x)\in\partial\Omega^t$ for all~$x\in\partial\Omega$, we have that
$$ u^t(\Phi^t(x))=0$$
for all~$x\in\partial\Omega$, and so, taking derivatives in~$t$,
\begin{equation}\label{HA63} \partial_t u^0(x)+<\nabla u^0(x),v(x)>\,=\,\partial_t u^t(x)+<
\nabla u^t(\Phi^t(x)),\partial_t\Phi^t(x)>\big|_{t=0}\,=\,0.\end{equation}
As a consequence,
\begin{eqnarray*}
&& \int_\Omega
\nabla u^0(x)\cdot\nabla \partial_t u^0(x)\,dx=
\int_\Omega
{\rm div} \big( \partial_t u^0(x)\nabla u^0(x)\big)\,dx-
\int_\Omega \partial_t u^0(x) f(x)\,dx
\\&&\qquad=
\int_{\partial\Omega}
\partial_t u^0(x)\,<\nabla u^0(x),\nu(x)>\,dS_x-
\int_\Omega \partial_t u^0(x) f(x)\,dx
\\&&\qquad=-
\int_{\partial\Omega}
<\nabla u^0(x),\nu(x)> \,<\nabla u^0(x),v(x)>\,dS_x-
\int_\Omega \partial_t u^0(x) f(x)\,dx,
\end{eqnarray*}
and therefore
\begin{equation}\label{4:1}
\begin{split}&
I'(0)=-
\int_{\partial\Omega}
<\nabla u^0(x),\nu(x)> \,<\nabla u^0(x),v(x)>\,dS_x\\&\qquad\qquad\qquad-
\int_\Omega \partial_t u^0(x) f(x)\,dx+
\frac12 \int_{\partial\Omega} |\nabla u^0(x)|^2\,<\nu(x),v(x)>\,dS_x.
\end{split}\end{equation}
We also remark that, since~$u^0=u=0$ on~$\partial\Omega$,
we have that
\begin{equation}\label{NA2}
\nu=\frac{\nabla u}{|\nabla u|}.\end{equation} Thus,
\begin{equation*}\begin{split}&
<\nabla u^0,\nu> \,<\nabla u^0,v>\,
=\,|\nabla u|^2\,<\nu,v>\\&\qquad\,=\,(
<\nabla u,\nu>)^2\,<\nu,v>\,=\,u_\nu^2\,<\nu,v>.
\end{split}\end{equation*}
In this way, \eqref{4:1} can be written as
\begin{equation}\label{4:2}
\begin{split}&
I'(0)=-
\int_\Omega \partial_t u^0 f\,dx-
\frac12 \int_{\partial\Omega} u_\nu^2\,<\nu,v>\,dS_x.
\end{split}\end{equation}
We also observe that~$\Delta\partial_tu^0=\partial_t\Delta u^0=\partial_t f=0$
in~$\Omega$, hence
\begin{eqnarray*}&&
\int_\Omega \partial_t u^0 f\,dx
=\int_\Omega \partial_t u^0 \Delta u\,dx
=\int_\Omega \big(\partial_t u^0 \Delta u-
\Delta\partial_tu^0\,u\big)\,dx\\
&&\qquad=\int_\Omega {\rm div}\,
\big(\partial_t u^0 \nabla u-
\nabla\partial_tu^0\,u\big)\,dx\\
&&\qquad=\int_{\partial\Omega}
<\partial_t u^0 \nabla u-
\nabla\partial_tu^0\,u,\,\nu>\,dS_x\\
&&\qquad=\int_{\partial\Omega}
\partial_t u^0 \,u_\nu \,dS_x.
\end{eqnarray*}
This, \eqref{HA63} and~\eqref{NA2}
entail that
\begin{eqnarray*}&&
\int_\Omega \partial_t u^0 f\,dx
=-\int_{\partial\Omega}
<\nabla u^0,v> \,u_\nu \,dS_x=
-\int_{\partial\Omega}
u_\nu^2<\nu,v> \,dS_x.
\end{eqnarray*}
Consequently, \eqref{4:2} becomes
\begin{equation*}
I'(0)=
\frac12 \int_{\partial\Omega} u_\nu^2\,<\nu,v>\,dS_x.
\end{equation*}
That is, being~$I$ stationary for all divergence free vector fields
is equivalent to
the constancy of~$u_\nu$, that is~\eqref{0-0-3}.

\end{appendix}

\section*{Acknowledgements}

The authors are supported by the Australian Research Council Discovery Project DP170104880 ``N.E.W. Nonlocal Equations at Work''
and are members of AustMS and INdAM/GNAMPA.
  
SD is supported by the DECRA Project DE180100957 ``PDEs, free boundaries and applications''.

\end{document}